\documentclass[
final
 , nomarks
]{dmtcs-episciences}
\usepackage{latexsym,amssymb,amsfonts,amsmath,url,tikz}
\usepackage{amssymb,tikz}
\usepackage[all]{xy}
\usepackage{verbatim}
\usepackage{accents}

\usepackage[english]{babel}
\usepackage{wrapfig,listings,textcomp}

\newcommand\thickbar[1]{\accentset{\rule{.4em}{.8pt}}{#1}}
\newtheorem{theorem}{Theorem}[section]
\newtheorem{proposition}[theorem]{Proposition}
\newtheorem{lemma}[theorem]{Lemma}
\newtheorem{corollary}[theorem]{Corollary}

\def\cfdf_theorem{\par\noindent{\bf Theorem~\ref{thm:fonderflaass}\ } \ignorespaces}
\def\endcfdf_theorem{}

\newtheorem{definition}[theorem]{Definition}
\newtheorem{example}[theorem]{Example}

\newtheorem{remark}[theorem]{Remark}
\numberwithin{equation}{section}
\def\I{\mathcal{IC}}
\def\IS{\mathcal{IS}}

\def\C{\mathcal{C}}
\def\A{\mathcal{A}}

\def\cc{\xi}
\def\cov{{\rm cov}}
\def\rem{{\rm rem}}
\DeclareMathOperator*{\pro}{Pro}

\usepackage[utf8]{inputenc}
\usepackage{subfigure}

\author{Jessica Striker\thanks{This work was partially supported by a National Security Agency Grant (H98230-15-1-0041), the North Dakota EPSCoR National Science Foundation Grant (IIA-1355466), the NDSU Advance FORWARD program sponsored by National Science Foundation grant (HRD-0811239), and a grant from the Simons Foundation/SFARI (527204, JS).}}
\title{Rowmotion and generalized toggle groups}
\affiliation{
North Dakota State University}
\received{2017-9-27}
\revised{2018-4-12}
\accepted{2018-5-7}
\begin{document}
\publicationdetails{20}{2018}{1}{17}{3962}
\maketitle
\begin{abstract}
We generalize the notion of the toggle group, as defined in [P.\ Cameron--D.\ Fon-der-Flaass '95] and further explored in [J.\ Striker--N.\ Williams '12], from the set of order ideals of a poset to any family of subsets of a finite set. We prove structure theorems for certain finite generalized toggle groups, similar to the theorem of Cameron and Fon-der-Flaass in the case of order ideals. We apply these theorems and find other results on generalized toggle groups in the following settings: chains, antichains, and interval-closed sets of a poset; independent sets, vertex covers, acyclic subgraphs, and spanning subgraphs of a graph; matroids and convex geometries. We generalize rowmotion, an action studied on order ideals in [P.\ Cameron--D.\ Fon-der-Flaass '95] and [J.\ Striker--N.\ Williams '12], to a map we call cover-closure on closed sets of a closure operator. We show that cover-closure is bijective if and only if the set of closed sets is isomorphic to the set of order ideals of a poset, which implies rowmotion is the only bijective cover-closure map.
\end{abstract}

%\thanks{Department of Mathematics, North Dakota State University, Fargo, ND
%\email{jessica.striker@ndsu.edu}

%Keywords
%Rowmotion, toggle group, poset, convex geometry

%AMS subject class
%05E18, 06A75

\maketitle

\section{Introduction}
In~\cite{fonderflaass}, P.~Cameron and D.~Fon-der-Flaass defined a permutation group on order ideals (or monotone Boolean functions) of a poset 
and proved that its group structure is the symmetric or alternating group whenever the Hasse diagram of the poset is connected.
In~\cite{prorow}, N.~Williams and the author named this permutation group the \emph{toggle group} and used the toggle group to prove results on the orbit structure of certain actions, finding bijective proofs for instances of the \emph{cyclic sieving phenomenon} of V.~Reiner, D.~Stanton, and D.~White~\cite{CSP}. 
In~\cite{fonderflaass}, Cameron and Fon-der-Flaass also studied a certain convex-closure action on order ideals and proved it could be expressed as the toggle group action in which one composes all the toggles of the poset from top to bottom. In~\cite{prorow}, N.~Williams and the author named this action \emph{rowmotion} and proved a theorem on
the equivariance of 
rowmotion and a toggle group action they called \emph{promotion}, loosely defined as toggling the elements of the poset from left to right. (See~\cite{prorow} for further history on rowmotion.) Then in~\cite{DilksPechenikStriker}, K.~Dilks, O.~Pechenik, and the author generalized this result from two to $n$ dimensions.
In~\cite{homomesy}, J.\ Propp and T.~Roby began investigating a phenomenon they called \emph{homomesy}. They built on the framework of~\cite{prorow} 
to prove instances of homomesy on the order ideals of the product of two chains poset under both rowmotion and promotion.
Subsequently, D.~Einstein and J.~Propp~\cite{EinsteinPropp} and D.~Grinberg and T.~Roby~\cite{Darij_Tom1, Darij_Tom2} extended the toggle group to the piecewise-linear and birational realms.

In this paper, we begin to generalize the above story beyond the setting of order ideals by noting that the particular structure of order ideals in a poset is unnecessary in the definition of the toggle group; the essential structure is  that an order ideal is a subset of poset elements. 
Thus, given a finite ground set $E$, we define the \emph{(generalized) toggle group} $T(\mathcal{L})$ on any family of subsets $\mathcal{L}\subseteq 2^E$; see Definitions~\ref{def:gentoggle} and~\ref{def:togglegroup}.

The papers cited above are part of
the emerging area of \emph{dynamical algebraic combinatorics}, defined as the study of dynamical systems arising from algebraic combinatorics, which was formalized at the American Institute of Mathematics (AIM) workshop on the topic~\cite{DACAIM}. 
Note that the generalized toggling framework introduced in this paper has already proved useful in papers stemming from the AIM workshop~\cite{EFGJMPRS, MJoseph}.

Some goals of studying the toggle group, in the order ideal setting and beyond, are the following.
\begin{enumerate}
\item Understand how to write cyclic group actions, such as rowmotion on order ideals, as toggle group elements.
\item Understand dynamical properties, such as orbit structure, cyclic sieving, and homomesy, of various toggle group elements.
\item Understand the abstract group structure of the toggle group, as Cameron and Fon-der-Flaass did for the order ideal toggle group.
\end{enumerate}
This paper makes contributions in all three areas, with a focus on the last.

 \smallskip
The main theorems of Section~\ref{sec:tog_construction}, highlighted below, contribute to Goal $(3)$ by giving conditions which guarantee the abstract group structure of a generalized toggle group is a direct product or contains the alternating group. (See Definitions~\ref{def:togdisjoint} and~\ref{def:pleasant} for definitions of $\oplus$, $\otimes$, and inductively toggle-alternating.) Recall $E$ is a finite set and $\mathcal{L}\subseteq 2^E$.

\begin{timestheorem} 
If $\mathcal{L}= \mathcal{L}_1\oplus\mathcal{L}_2$ or $\mathcal{L}= \mathcal{L}_1\otimes\mathcal{L}_2$, then $T(\mathcal{L}) = T(\mathcal{L}_1)\times T(\mathcal{L}_2)$.
\end{timestheorem}

\begin{togaltdef} 
We say $\mathcal{L}$ is \emph{toggle-alternating} if $T(\mathcal{L})$ contains the alternating group ${A}_{\mathcal{L}}$.
\end{togaltdef}

\begin{maintheorem} \it
If $\mathcal{L}$ is inductively toggle-alternating, then it is toggle-alternating.
\end{maintheorem}

The proofs follow the proof of Cameron and Fon-der-Flaass~\cite{fonderflaass} for order ideal toggle groups. The essence of these theorems is that they isolate the qualities of the set of order ideals which produce this toggle group description, so that we may apply the same structure description to other situations.

One may wonder whether it is useful to know that a toggle group is the symmetric or alternating group and how one might determine which of the two possibilities it is. In the case of the order ideal toggle group, Cameron and Fon-der-Flaass found no simple criterion for determining whether a toggle group is symmetric or alternating~\cite{fonderflaass}. But they did observe that in many cases, such as when there is a unique maximal or minimal element in the poset, one may find a toggle which is a transposition. Such a toggle group must be the symmetric group, since it contains an odd permutation, and thus cannot be the alternating group.
The abstract structure of the toggle group may also be important when looking for bijections between subsets of $\mathcal{L}$, since if $T(\mathcal{L})$ is $\mathfrak{S}_{\mathcal{L}}$, this guarantees the desired permutation of elements can be realized as a product of toggles, while if $T(\mathcal{L})$ is $A_{\mathcal{L}}$, the permutation may be realized only if it is even. See Remark~\ref{remark:bij} for further discussion. Finally, such toggle groups give presentations for the symmetric and alternating groups that may be of independent interest.

\smallskip
We apply Theorems~\ref{thm:times} and~\ref{thm:toggpstructure} to various combinatorially interesting sets of subsets to obtain the following corollary, which is actually a compilation of corollaries proved in Section~\ref{sec:ex}.

\begin{maincorollary}
Toggle groups of the following sets have a group structure given by Theorems~\ref{thm:times} and~\ref{thm:toggpstructure}:
\begin{itemize}
\item Order ideals of a finite poset (proven in~\cite[Theorem 4]{fonderflaass}),
\item Chains of a finite poset (see Section~\ref{sec:chains}),
\item Antichains of a finite poset (see Section~\ref{sec:antichains}),
\item Interval-closed sets of a certain family of finite posets (see Section~\ref{sec:ic}),
\item Independent sets of a finite graph (see Section~\ref{sec:indepsets}), and
\item Vertex covers of a finite graph (see Section~\ref{sec:vertexcovers}).
\end{itemize}
\end{maincorollary}

\noindent
Notable other results/connections of Section~\ref{sec:ex} include:
\begin{itemize}
\item Theorems~\ref{thm:chaineq} and \ref{thm:antichaineq}, which prove the equivariance of some toggle orders in the chain and antichain toggle groups, respectively (this falls under Goal $(2)$ above and is analogous to the main result of~\cite{prorow} on the equivariance of promotion and rowmotion);
\item Toggle commutation lemmas in all of the toggle groups listed in Corollary~\ref{cor:mastercor}, as well as the acyclic subgraph toggle group, the spanning subgraph toggle group, and the matroid independent set toggle group;
\item Characterizations of matroids and convex geometries in terms of toggles (Remarks~\ref{remark:matroid1} and~\ref{remark:cxgeom}).
\item Remarks on connections between specific generalized toggle groups and research in statistical physics and commutative algebra (Remarks~\ref{remark:indep} and \ref{remark:vc});
\end{itemize}

\smallskip
In Section~\ref{sec:row}, we generalize rowmotion on order ideals to a map we call \emph{cover-closure}.

\begin{maindefinition} 
Let $E$ be a  set with closure operator $\tau$ whose set of closed sets is $\mathcal{L}\subseteq 2^E$. 
For $X\in 2^E$, %let $\cov(X)\subseteq E\setminus X$ be  
define $\cov(X)=\{e\in E\setminus X \ | \ X\cup\{e\}\in\mathcal{L} \}$.
%the maximal subset of $E\setminus X$ such that $\forall e\in \cov(X)$, $X\cup\{e\}\in\mathcal{L}$. 
Call $\cov(X)$ the set of \emph{covers} of $X$.
Then we define \emph{cover-closure} $\cc:\mathcal{L} \rightarrow \mathcal{L}$ as $\cc(X)=\tau(\cov(X))$.
\end{maindefinition}

When restricted to order ideals of a poset, we show in Lemma~\ref{lem:ccrow} cover-closure is the usual rowmotion operation, defined in~\cite{fonderflaass} and elsewhere as the order ideal generated by the minimal elements of the poset which are not already in the order ideal. 

%\smallskip
Our {second main theorem}, stated below, gives a `negative' result in the direction of Goal $(1)$. It characterizes bijective cover-closure maps, yielding Corollary~\ref{cor:rowcc} that rowmotion is the only bijective cover-closure map.

\begin{ccthm}
Given a finite ground set $E$ and a closure operator $\tau$ with set of closed sets $\mathcal{L}\subseteq 2^E$,
cover-closure $\cc:\mathcal{L}\rightarrow\mathcal{L}$ is bijective if and only if 
$\mathcal{L}$ is isomorphic to
the set of order ideals $J(P)$ for some poset~$P$.
\end{ccthm}

The paper is structured as follows. Section~\ref{sec:tog_construction} defines generalized toggle groups and proves Theorems \ref{thm:times} and~\ref{thm:toggpstructure}, which together give a structure description of the toggle group for certain families. 
Section~\ref{sec:ex} gives many examples of generalized toggle groups, proves Corollary~\ref{cor:mastercor}, and gives  many further results and connections. Finally, Section~\ref{sec:row} defines cover-closure, relates it to rowmotion, and proves Theorem~\ref{thm:row}.

\section{Construction and structure of generalized toggle groups}
\label{sec:tog_construction}
For everything that follows, let $E$ be a finite\footnote{We could take $E$ to be infinite in our definitions, but since all our results concern finite ground sets, we restrict our attention to this case.} set and $\mathcal{L}$ be any subset of the power set $2^{E}$. 
\begin{definition}
\label{def:gentoggle}
For each element $e\in E$ define its \emph{toggle} $t_e:\mathcal{L}\rightarrow\mathcal{L}$ as follows.
\[
 t_e(X) =
  \begin{cases}
   X\cup\{e\} & \text{if } e\notin X \text{ and } X\cup\{e\}\in\mathcal{L} \\
   X\setminus\{e\} & \text{if } e\in X \text{ and } X\setminus\{e\}\in\mathcal{L} \\
   X       & \text{otherwise.}
  \end{cases}
\]
We call $\left\{t_e \ | \ e\in E\right\}$ the set of \emph{toggles} of $E$.
\end{definition}

\begin{remark} \rm 
Note that $t_e^2=1$ for all $e\in E$. 
\end{remark}

We define the (generalized) toggle group as the group generated by these toggles.
\begin{definition}
\label{def:togglegroup}
Let $T(\mathcal{L})$ be the subgroup of the symmetric group $\mathfrak{S}_{\mathcal{L}}$, generated by $\left\{t_e \ | \ e\in E\right\}$. Call $T(\mathcal{L})$ the \emph{(generalized) toggle group} on $\mathcal{L}$.
\end{definition}

\begin{example} \rm
\label{ex:first}
Let $E=\left\{1,2,3,4\right\}$ and $\mathcal{L}=\{\emptyset, \{1\}, \{1,2\}, \{1,3\}, \{1,2,3\},$ $\{1,2,3,4\}\}$.

Then, for example,

\begin{center}
$t_4(\{1,2,3\})= \{1,2,3,4\}$, \hspace{.5cm} 
$t_2(\{1,2\})= \{1\}$, \hspace{.5cm}
$t_2(\{1,2,3,4\})= \{1,2,3,4\}$.
\end{center}

We can write $t_1, t_2, t_3, t_4$ as permutations in cycle notation (with 1-cycles omitted) as follows.
\begin{align*}
t_1 &=(\emptyset, \{1\}) \\
t_2 &=\left(\{1\}, \{1,2\})(\{1,3\}, \{1,2,3\}\right)\\
t_3 &=(\{1\}, \{1,3\})(\{1,2\}, \{1,2,3\})\\
t_4 &=(\{1,2,3\}, \{1,2,3,4\})
\end{align*}
\end{example}

We give the following definitions to provide language to describe when two families of subsets are ``isomorphic.''
\begin{definition}
Given $\mathcal{L}\subseteq 2^E$, let an \emph{essentialization} of $(\mathcal{L},E)$ be $(\mathcal{L}',E')$ where $E'$ and $\mathcal{L}'$ are as follows. Every $e\in E'$ must be an element $e\in E$ such that there exist $X,Y\in\mathcal{L}$ with $e\in X$ and $e\not\in Y$. That is, no element of $E'$ may be contained in all or none of the subsets in $\mathcal{L}$. Also, suppose   $Z\subseteq E, |Z|>1$ such that if $e\in Z$ and $e\in W\in\mathcal{L}$, then $Z\subseteq W$. For any maximal such $Z$, we exclude all elements of $Z$ except one from $E'$. Let $\mathcal{L}'=\{X\cap E' \ | \ X\in\mathcal{L}\}$. 
\end{definition}

\begin{example}
Let $E=\left\{1,2,3,4,5,6\right\}$ and $\mathcal{L}=\{\{1\}, \{1,2,5,6\}, \{1,3,5,6\}\}$. Then in the essentialization, we remove $1$ and $4$, since $1$ is in every subset and $4$ is in none of the subsets. Also, since $5$ and $6$ always appear together, we reduce $\{5,6\}$ to the single element $5$ and remove $6$. Therefore,  $E'=\{2,3,5\}$ and $\mathcal{L}'=\{\emptyset,\{2,5\},\{3,5\}\}$.
\end{example}

\begin{definition}
Say $\mathcal{L}_1\subseteq 2^{E_1}$ and $\mathcal{L}_2\subseteq 2^{E_2}$ are \emph{isomorphic} if there exists a bijection $\varphi:E_1'\rightarrow E_2'$ such that $\varphi(\mathcal{L}_1')=\mathcal{L}_2'$.
\end{definition}

\begin{remark} \rm
Note the structure of the toggle group depends only on the isomorphism type of the family of subsets. In particular, $T(\mathcal{L})= T(\mathcal{L}')$ for any essentialization $(\mathcal{L}',{E}')$.
\end{remark}

\smallskip
We now define the \emph{toggle poset}, which we will use in Theorem~\ref{thm:row}; see Figures~\ref{fig:togposetex} and~\ref{fig:not_sn_an} for examples.

\begin{definition}
\label{def:Lposet}
The \emph{toggle poset} $\mathcal{P}_{\mathcal{L}}$ of $\mathcal{L}$ is defined as the partial order on $\mathcal{L}$ where the covering relations are given by $X\lessdot Y$ if and only if $X\subseteq Y$ and $|Y\setminus X|=1$, that is, if there exists $e\in E$ such that $t_e(X)=Y$. 
\end{definition}

\begin{figure}[htbp]
\begin{center}
\begin{tikzpicture}[>=latex,line join=bevel,]
\node (node_5) at (39.5bp,220.5bp) [draw,draw=none] {$\left\{1, 2, 3, 4\right\}$};
  \node (node_4) at (15.5bp,114.5bp) [draw,draw=none] {$\left\{1, 2\right\}$};
  \node (node_3) at (39.5bp,8.5bp) [draw,draw=none] {$\left\{\right\}$};
  \node (node_2) at (39.5bp,61.5bp) [draw,draw=none] {$\left\{1\right\}$};
  \node (node_1) at (64.5bp,114.5bp) [draw,draw=none] {$\left\{1, 3\right\}$};
  \node (node_0) at (39.5bp,167.5bp) [draw,draw=none] {$\left\{1, 2, 3\right\}$};
  \draw [black,->] (node_3) ..controls (39.5bp,23.805bp) and (39.5bp,34.034bp)  .. (node_2);
  \draw [black,->] (node_0) ..controls (39.5bp,182.81bp) and (39.5bp,193.03bp)  .. (node_5);
  \draw [black,->] (node_2) ..controls (46.623bp,77.031bp) and (51.863bp,87.72bp)  .. (node_1);
  \draw [black,->] (node_1) ..controls (57.377bp,130.03bp) and (52.137bp,140.72bp)  .. (node_0);
  \draw [black,->] (node_4) ..controls (22.338bp,130.03bp) and (27.368bp,140.72bp)  .. (node_0);
  \draw [black,->] (node_2) ..controls (32.662bp,77.031bp) and (27.632bp,87.72bp)  .. (node_4);
\end{tikzpicture}
\end{center}
\caption{The toggle posets $\mathcal{P}_{\mathcal{L}}$, where  $\mathcal{L}$ is as in Example~\ref{ex:first}.}
\label{fig:togposetex}
\end{figure}
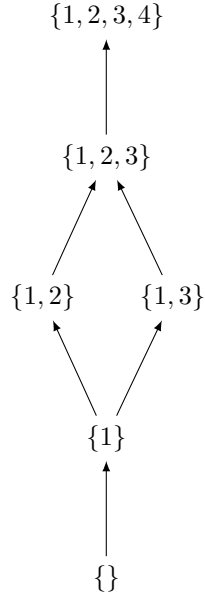

In the following proposition, we give some properties of the toggle poset. We will need the following definition.

\begin{definition}
A poset ${P}$ is \emph{graded} if there exists a rank function $\rho:{P}\rightarrow \mathbb{Z}$ such that if $y$ covers $x$ in ${P}$, $\rho(x)=\rho(y)+1$. A poset is \emph{strongly graded} if all maximal chains have the same length. 
\end{definition}

\begin{proposition}
\label{prop:togposet}
$\mathcal{P}_{\mathcal{L}}$ has the following properties:
\begin{enumerate}
\item $\mathcal{P}_{\mathcal{L}}$ is graded but not necessarily strongly graded. 
\item If the edges of the Hasse diagram of $\mathcal{P}_{\mathcal{L}}$ are labeled by the appropriate toggle, any toggle may appear at most once in any chain.
\end{enumerate}
\end{proposition}

\begin{proof}
To show part (1), note there is a rank function $\rho:\mathcal{P}_{\mathcal{L}}\rightarrow \mathbb{Z}$ given by $\rho(X)=|X|$ and covering relations only appear between elements on adjacent ranks. Thus $\mathcal{P}_{\mathcal{L}}$ is graded. 
To show $\mathcal{P}_{\mathcal{L}}$ might not be strongly graded, consider as an example $E=\{1,2,3\}$ and $\mathcal{L}=\left\{\emptyset, \{1\}, \{2\},\{1,3\}\right\}$. Both $\{\emptyset, \{2\}\}$ and $\{\emptyset, \{1\}, \{1,3\}\}$ are maximal chains, but they do not have the same length. %Since maximal chains need not have the same length,  $\mathcal{P}_{\mathcal{L}}$ is not strongly graded. 

For part (2), note that each cover adds a single element to the subset and removes no elements. No element may be added twice, so each Hasse diagram edge in any chain must represent a different toggle. 
\end{proof}

%We now remark on some properties that $\mathcal{P}_{\mathcal{L}}$ does not satisfy in general.
\begin{remark}
$\mathcal{P}_{\mathcal{L}}$ is {not always} the same as partially ordering $\mathcal{L}$ by containment, in fact, $\mathcal{P}_{\mathcal{L}}$ might not be a lattice, and the Hasse diagram of $\mathcal{P}_{\mathcal{L}}$ need not even be connected.
Consider as an example $E=\{1,2,3\}$ and $\mathcal{L}=\left\{\{1\}, \{2\}, \{1,2,3\}\right\}$. In $\mathcal{P}_{\mathcal{L}}$, $\{1\}$ and $\{1,2,3\}$ are incomparable, since there is no single toggle that maps $\{1\}$ to $\{1,2,3\}$. Also, the Hasse diagram of $\mathcal{P}_{\mathcal{L}}$ in this example is disconnected (three isolated points). Therefore, $\mathcal{P}_{\mathcal{L}}$ is not a lattice. 
\end{remark}

\begin{example} \rm
\label{ex:togposetnote}
In the case where $\mathcal{L}$ equals the set of order ideals of a finite poset $P$, the toggle poset $\mathcal{P}_{\mathcal{L}}$ is the distributive lattice of order ideals $J(P)$. Thus, in contrast to the general case  discussed above, $\mathcal{P}_{\mathcal{L}}=J(P)$ has connected Hasse diagram, is strongly graded, and is equivalent to the partial order by containment.
\end{example}

In the order ideal case $\mathcal{L}=J(P)$, P.~Cameron and D.~Fon-der-Flaass proved the following structure theorem for $T(\mathcal{L})$. 

\begin{theorem}[\protect{\cite[Theorem 4]{fonderflaass}}]
\label{thm:fonderflaass}
Let $P$ be a finite poset and $J(P)$ the set of order ideals of $P$. If $P$ is not the disjoint union of  posets, $T(J(P))$ is either the symmetric group $\mathfrak{S}_{J(P)}$ or the alternating group $A_{J(P)}$. If $P$ is the disjoint union of two posets $P=P_1+ P_2$, then $T(J(P)) = T(J(P_1))\times T(J(P_2))$.
\end{theorem}

Initial computations (including all toggle groups with $|E|\leq 3$ 
and many other larger examples) indicated that all generalized toggle groups might be a direct product of symmetric and/or alternating groups as in the case of order ideals; however,
this is not true, as shown by the following example. 

\begin{example} \rm
\label{ex:not_sn_an}
Let $E=\{1,2,3,4\}$ and $\mathcal{L}=\{\emptyset,\{1\},\{1,2\},\{1,2,3\},\{1,2,3,4\},\{2,3,4\},\{3,4\},\{4\}\}$. 
We have computed using \verb|SageMath|~\cite{sage} that the toggle group $T(\mathcal{L})$ as a subgroup of $\mathfrak{S}_{\mathcal{L}}\cong \mathfrak{S}_8$ (via the isomorphism in which the $i$th element of $\mathcal{L}$ is mapped to $i$) is generated by the following four permutations: 
\[\left\langle ~ (1~2) ( 5~6), (2~3) (6~7), (3~4)(7~8) , (1~8) ( 4~5) ~\right\rangle.\]
 We used \verb|GAP|~\cite{gap} via \verb|SageMath|~\cite{sage} to compute that $T(\mathcal{L})=((((C_2\times D_4)\rtimes C_2)\rtimes C_3)\rtimes C_2)$ where $C_n$ is the cyclic group of order $n$ and $D_4$ is the dihedral group of order $8$. See Figure~\ref{fig:not_sn_an} for the toggle poset of this example and Appendix~\ref{sec:appendix} for the \verb|SageMath| code to run these computations.
\end{example}

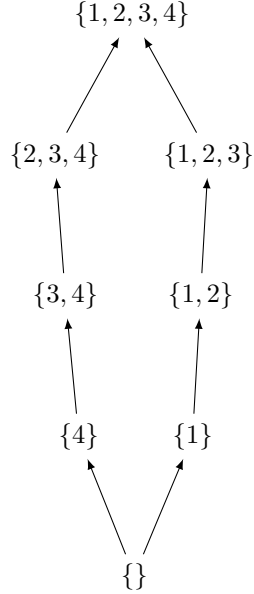
\begin{figure}[htbp]
\begin{center}
\begin{tikzpicture}[>=latex,line join=bevel,]
\node (node_7) at (49.0bp,220.5bp) [draw,draw=none] {$\left\{1, 2, 3, 4\right\}$};
  \node (node_6) at (75.0bp,114.5bp) [draw,draw=none] {$\left\{1, 2\right\}$};
  \node (node_5) at (50.0bp,8.5bp) [draw,draw=none] {$\left\{\right\}$};
  \node (node_4) at (20.0bp,167.5bp) [draw,draw=none] {$\left\{2, 3, 4\right\}$};
  \node (node_3) at (24.0bp,114.5bp) [draw,draw=none] {$\left\{3, 4\right\}$};
  \node (node_2) at (29.0bp,61.5bp) [draw,draw=none] {$\left\{4\right\}$};
  \node (node_1) at (72.0bp,61.5bp) [draw,draw=none] {$\left\{1\right\}$};
  \node (node_0) at (78.0bp,167.5bp) [draw,draw=none] {$\left\{1, 2, 3\right\}$};
  \draw [black,->] (node_2) ..controls (27.598bp,76.805bp) and (26.595bp,87.034bp)  .. (node_3);
  \draw [black,->] (node_6) ..controls (75.841bp,129.81bp) and (76.443bp,140.03bp)  .. (node_0);
  \draw [black,->] (node_4) ..controls (28.306bp,183.11bp) and (34.472bp,193.95bp)  .. (node_7);
  \draw [black,->] (node_5) ..controls (44.048bp,23.956bp) and (39.71bp,34.49bp)  .. (node_2);
  \draw [black,->] (node_5) ..controls (56.268bp,24.031bp) and (60.879bp,34.72bp)  .. (node_1);
  \draw [black,->] (node_0) ..controls (69.694bp,183.11bp) and (63.528bp,193.95bp)  .. (node_7);
  \draw [black,->] (node_3) ..controls (22.878bp,129.81bp) and (22.076bp,140.03bp)  .. (node_4);
  \draw [black,->] (node_1) ..controls (72.841bp,76.805bp) and (73.443bp,87.034bp)  .. (node_6);
\end{tikzpicture}
\end{center}
\caption{The toggle poset $\mathcal{P}_{\mathcal{L}}$, where $\mathcal{L}$ is as in Example~\ref{ex:not_sn_an}.}
\label{fig:not_sn_an}
\end{figure}

Even though the above example shows that generalized toggle group structure is not well-behaved in general, we show in this paper that similar structure theorems to Theorem~\ref{thm:fonderflaass} do exist in many other {combinatorially significant} generalized toggle groups. Our first main results, Theorems~\ref{thm:times} and~\ref{thm:toggpstructure}, give sufficient conditions for $T(\mathcal{L})$ to have a structure theorem similar to Theorem~\ref{thm:fonderflaass}. We then summarize in Corollary~\ref{cor:mastercor} (and prove in Section~\ref{sec:ex}) that the following cases satisfy these conditions: chains, antichains, and interval-closed sets of a poset; independent sets and vertex covers of a graph.

\smallskip
We begin with Theorem~\ref{thm:times} discussing cases in which the toggle group reduces to a direct product. We will need the following definition.
\begin{definition}
\label{def:togdisjoint}
$\mathcal{L}$ is a \emph{toggle-disjoint direct sum} of $\mathcal{L}_1$ and $\mathcal{L}_2$, written $\mathcal{L}= \mathcal{L}_1\oplus\mathcal{L}_2$, if $\mathcal{L}=\mathcal{L}_1\cup \mathcal{L}_2$  and their essentializations $(\mathcal{L}_1',E_1')$ and $(\mathcal{L}_2',E_2')$ satisfy $E_1'\cap E_2'=\emptyset$.
$\mathcal{L}$ is a \emph{toggle-disjoint Cartesian product} of $\mathcal{L}_1$ and $\mathcal{L}_2$, written $\mathcal{L}= \mathcal{L}_1\otimes\mathcal{L}_2$, if $\mathcal{L}=\{X_1\cup X_2 \ | \ X_1\in\mathcal{L}_1, X_2\in\mathcal{L}_2\}$  and their essentializations $(\mathcal{L}_1',E_1')$ and $(\mathcal{L}_2',E_2')$ satisfy $E_1'\cap E_2'=\emptyset$.
\end{definition}

\begin{remark} \rm
$\mathcal{L}= \mathcal{L}_1\oplus\mathcal{L}_2$ implies $\mathcal{P}_{\mathcal{L}}= \mathcal{P}_{\mathcal{L}_1} + \mathcal{P}_{\mathcal{L}_2}$, where $+$ denotes the disjoint union of posets. $\mathcal{L}= \mathcal{L}_1\otimes\mathcal{L}_2$ implies $\mathcal{P}_{\mathcal{L}}= \mathcal{P}_{\mathcal{L}_1} \times \mathcal{P}_{\mathcal{L}_2}$ where $\times$ denotes the Cartesian product of posets.
\end{remark}

\begin{theorem}
\label{thm:times}
If $\mathcal{L}= \mathcal{L}_1\oplus\mathcal{L}_2$ or $\mathcal{L}= \mathcal{L}_1\otimes\mathcal{L}_2$, then $T(\mathcal{L}) = T(\mathcal{L}_1)\times T(\mathcal{L}_2)$.
\end{theorem}

\begin{proof} 
Suppose ${\mathcal{L}}$ is the toggle-disjoint direct sum $\mathcal{L}= \mathcal{L}_1\oplus\mathcal{L}_2$. So then the essentialized ground sets $E_1',E_2'\subset E$ satisfy $E_1'\cap E_2'=\emptyset$. 
Then $t_{e_1}t_{e_2}=t_{e_2}t_{e_1}$ for all $e_1\in E_1'$ and $e_2\in E_2'$, since toggles in $E_1'$ act trivially on all elements of $\mathcal{L}_2'$ and vice versa. Thus $T(\mathcal{L})= T(\mathcal{L}_1')\times T(\mathcal{L}_2')= T(\mathcal{L}_1)\times T(\mathcal{L}_2)$.

Suppose ${\mathcal{L}}$ is the toggle-disjoint Cartesian product $\mathcal{L}= \mathcal{L}_1\otimes\mathcal{L}_2$. Again,  the essentialized ground sets $E_1',E_2'\subset E$ satisfy $E_1'\cap E_2'=\emptyset$.  We may represent the elements of ${\mathcal{L}}$ as ordered pairs $(X_1',X_2')$ where $X_1'\in {\mathcal{L}_1'}, X_1'\subseteq E_1'$ and $X_2'\in {\mathcal{L}_2'}, X_2'\subseteq E_2'$. 
The toggle groups $T(\mathcal{L}_1')$ and $T(\mathcal{L}_2')$ act on their respective components, so we have  $T(\mathcal{L})= T(\mathcal{L}_1')\times T(\mathcal{L}_2')= T(\mathcal{L}_1)\times T(\mathcal{L}_2)$.
\end{proof}

We next prove Theorem~\ref{thm:toggpstructure}, which gives conditions under which a generalized toggle group contains the alternating group. We give this a name in the definition below. 

\begin{definition}
\label{def:togaltdef}
We say $\mathcal{L}$ is \emph{toggle-alternating} if $T(\mathcal{L})$ contains the alternating group ${A}_{\mathcal{L}}$.
\end{definition}

We will use the following additional definitions. 

\begin{definition}
Given $e\in E$ and $\mathcal{L}\subseteq 2^E$, define $\mathcal{L}_e:=\{X\in\mathcal{L} \ | \ e\in X\}$ to consist of all the subsets in $\mathcal{L}$ that contain $e$ and $\mathcal{L}_{\thickbar{e}}:=\{X\in\mathcal{L} \ | \ e\notin X\}$  to be all the subsets in $\mathcal{L}$ that do not contain $e$. For any subset $S\subseteq \mathcal{L}$, let $t_e(S)=\left\{t_e(X) \mid X\in S\right\}$. 
\end{definition}

\begin{definition}
\label{def:pleasant}
We define the collection of \emph{inductively toggle-alternating} sets of subsets to be the smallest collection of $\mathcal{L}\subseteq 2^E$, with essentializations $(\mathcal{L}',E')$, such that:
\begin{enumerate}
\item 
if $|E'|\leq 4$, then $T(\mathcal{L})$ is toggle-alternating;
\item if $|E'|\geq 5$, 
then there exists $e\in E'$ such that at least one of the following holds:
\begin{enumerate}
\item
$\mathcal{L}_e$ is inductively toggle-alternating, $\mathcal{L}_e\cup t_e(\mathcal{L}_e) = \mathcal{L}$, and $\mathcal{L}_e\cap t_e(\mathcal{L}_e) \neq \emptyset$, or
\item $\mathcal{L}_{\thickbar{e}}$ is inductively toggle-alternating, $\mathcal{L}_{\thickbar{e}}\cup t_e(\mathcal{L}_{\thickbar{e}}) = \mathcal{L}$, and $\mathcal{L}_{\thickbar{e}}\cap t_e(\mathcal{L}_{\thickbar{e}}) \neq \emptyset$.
\end{enumerate}
\end{enumerate}
\end{definition}

\begin{theorem}
\label{thm:toggpstructure}
If $\mathcal{L}$ is inductively toggle-alternating, then it is toggle-alternating.
\end{theorem}

\begin{proof}
Let $\mathcal{F}$ be the collection of inductively toggle-alternating sets of subsets of $E$.
Suppose $\mathcal{L}\in\mathcal{F}$ and recall $(\mathcal{L}',E')$ denotes the essentialization of $(\mathcal{L},E)$.
The proof follows the proof of Cameron and Fon-der-Flaass~\cite{fonderflaass} for the order ideal toggle group and is by induction on the size of $E'$. 

Note for $|E'|\leq 4$, we have required in the definition of inductively toggle-alternating that $T(\mathcal{L}')$ be toggle-alternating. This is because in the proof below we need $A_{\mathcal{L}}$ to be simple, and recall the alternating group ${A}_n$ is simple for $n\geq 5$ but not for $n=4$. Thus $|E'|\leq 4$ form the base cases for our induction.

Now suppose $|E'|\geq 5$. Also, suppose for all $\mathcal{L}_0\in\mathcal{F}$ with $|E_0'|<|E'|$ we have $\mathcal{L}_0$ is toggle-alternating.
Let $e\in E'$ be the element required by condition $(2)$ of Definition~\ref{def:pleasant}, since $\mathcal{L}$ is inductively toggle-alternating.
So we know $\mathcal{L}$ is the disjoint union of $\mathcal{L}_e$ and $\mathcal{L}_{\thickbar{e}}$, and neither $\mathcal{L}_e$ nor $\mathcal{L}_{\thickbar{e}}$ are empty. Suppose $\mathcal{L}$ satisfies condition $(2\mathrm{a})$ of Definition~\ref{def:pleasant}, so that $\mathcal{L}_e$ is inductively toggle-alternating and $|\mathcal{L}_e|\geq|\mathcal{L}_{\thickbar{e}}|$ (if not, then $\mathcal{L}$ satisfies condition $(2\mathrm{b})$, so $|\mathcal{L}_e|\leq|\mathcal{L}_{\thickbar{e}}|$, and the argument follows symmetrically with $\mathcal{L}_{\thickbar{e}}$ instead). 
Now since all the elements of $\mathcal{L}_e$ contain $e$,
$t_e$ will act as the identity 
in $T(\mathcal{L}_e)$.
That is, the essentialization $\mathcal{L}'_e$ has ground set at most $E'\setminus \{e\}$. 

By induction, $T(\mathcal{L}_e)$ contains the alternating group $A_{\mathcal{L}_e}$. 
Assume $|\mathcal{L}_e|\geq 5$ so that $A_{\mathcal{L}_e}$ is simple. Since $A_{\mathcal{L}_e}$ 
is simple and $|\mathcal{L}_e|>|\mathcal{L}_{\thickbar{e}}|$, 
the subgroup $K$ of $T(\mathcal{L})$ fixing $\mathcal{L}_{\thickbar{e}}$ pointwise induces at least $A_{\mathcal{L}_e}$ on $\mathcal{L}_e$. In the case $|\mathcal{L}_e|\leq 4$, 
the essentialized ground set of $\mathcal{L}_e$, which we denote $E_e'$, satisfies 
$|E_{e}'|\leq 4$ and condition $(1)$ in Definition~\ref{def:pleasant} gives that the subgroup $K$ of $T(\mathcal{L})$ fixing $\mathcal{L}_{\thickbar{e}}$ pointwise induces at least $A_{\mathcal{L}_e}$ on $\mathcal{L}_e$. Then in either case,
$\mathcal{L}_e\cup t_e(\mathcal{L}_e) = \mathcal{L}$ 
and $\mathcal{L}_e\cap t_e(\mathcal{L}_e) \neq \emptyset$,
so $K$ and the conjugate subgroup $K^{t_e}$ ($:=\left\{t_e k t_e \ | \ k\in K\right\}$) generate an alternating or symmetric group on~$\mathcal{L}$. 
\end{proof}

In the following corollary, we apply Theorems~\ref{thm:times} and~\ref{thm:toggpstructure} to 
obtain a toggle group structure description for various combinatorially interesting sets of subsets.
We prove this corollary in each case in Section~\ref{sec:ex} by the following method. We first discuss cases in which ${\mathcal{L}}$ is a toggle-disjoint direct sum or Cartesian product and use Theorem~\ref{thm:times} to conclude $T(\mathcal{L})$ is a direct product. 
Secondly, we show the remaining family of subsets is inductively toggle-alternating by exhibiting an $e\in E$ satisfying condition (2) in Definition~\ref{def:pleasant} and proving the base cases $|E|\leq 4$ (condition $(1)$ in Definition~\ref{def:pleasant}). See Appendix~\ref{sec:appendix} for the \verb|SageMath| code used to check these base cases.

\begin{corollary}
\label{cor:mastercor}
Toggle groups of the following sets have a group structure given by Theorems~\ref{thm:times} and~\ref{thm:toggpstructure}:
\begin{itemize}
\item Order ideals of a poset (proven in~\cite{fonderflaass}, see Section~\ref{sec:orderideals}),
\item Chains of a poset (see Section~\ref{sec:chains}),
\item Antichains of a poset (see Section~\ref{sec:antichains}),
\item Interval-closed sets of a certain family of posets (see Section~\ref{sec:ic}),
\item Independent sets of a graph (see Section~\ref{sec:indepsets}), and
\item Vertex covers of a graph (see Section~\ref{sec:vertexcovers}).
\end{itemize}
\end{corollary}

In Section~\ref{sec:ex}, we discuss these toggle groups in greater detail and prove some additional results. We also consider several other examples in addition to the above list, namely, unions of subset families (Section~\ref{sec:multipleposets}), acyclic subgraphs of a graph (Section~\ref{sec:acyclic}), connected subgraphs of a graph (Section~\ref{sec:conn}), matroids (Section~\ref{sec:matroids}), and convex geometries  (Section~\ref{sec:antimatroids}).

\section{Examples of combinatorially significant generalized toggle groups}
\label{sec:ex}
In this section, we look at many examples of toggle groups on combinatorial objects which are given as subsets of some finite ground set. 
 For many of the examples, we prove a lemma concerning when toggles commute and give a structure theorem.
The toggle group on order ideals $J(P)$ of a poset $P$, as studied in~\cite{fonderflaass,prorow}, serves as the motivating example; we list some relevant properties and theorems in Section~\ref{sec:orderideals} for the sake of comparison. 
This is neither an exhaustive list nor a complete treatment of any of these examples. Rather, we aim to identify these generalized toggle groups as avenues for further study and initiate the investigation in each case. 

Sections \ref{sec:orderideals}--\ref{sec:ic} study various poset-theoretic toggle groups. (See~\cite[Chapter 3]{Stanley_1} for background on posets.) Sections \ref{sec:indepsets}--\ref{sec:conn} examine toggle groups on graphs. Section~\ref{sec:matroids} considers matroids and Section~\ref{sec:antimatroids} discusses convex geometries.
\subsection{Order ideals}
\label{sec:orderideals}
Given a poset $P$, an \emph{order ideal} $I$ is a subset $I\subseteq P$ such that if $y\in I$ and $x\leq y$ then $x\in I$. Denote as $J(P)$ the set of order ideals of $P$. Since order ideals are subsets of the elements of $P$, we take $P$ as our ground set and $\mathcal{L}=J(P)\subseteq 2^P$.

In this subsection, we state the toggle commutation lemma and toggle group structure theorem of Cameron and Fon-der-Flaass~\cite{fonderflaass} for the order ideal toggle group $T(J(P))$. We give proofs as motivating examples for subsequent subsections. We also mention some related results on rowmotion.

\begin{lemma}[\protect{\cite[p.\ 546]{fonderflaass}}] 
\label{lem:oicommute}
In the order ideal toggle group $T(J(P))$, $(t_p t_q)^2=1$ if and only if there is no covering relation between $p$ and $q$.
\end{lemma}

\begin{proof}
Suppose $q$ covers $p$. Let $I$ be the minimal order ideal containing every poset element that $q$ covers. Then $t_p t_q(I)=I\cup \{q\}$ whereas $t_q t_p(I)=I\setminus \{p\}$. So $t_p$ and $t_q$ do not commute. 

If $p<q$ but there exists $r$ such that $p<r<q$, then for any $I\in J(P)$, if $r\in I$, $t_p(I)=I$, and if $r\notin I$, then $t_q(I)=I$. So $t_p$ and $t_q$ commute.

If $p$ and $q$ are incomparable, then the presence or absence of $p$ in $I$ does not affect whether $q$ may be in $I$. So $t_p$ and $t_q$ commute.
\end{proof}

We stated as Theorem~\ref{thm:fonderflaass} a result of Cameron and Fon-der-Flaass~\cite{fonderflaass} on the order ideal toggle group structure. We give a proof below using Theorems~\ref{thm:times} and~\ref{thm:toggpstructure}, whose proofs were inspired by their original proof.
%\begin{cfdf_theorem}[\protect{\cite[Theorem 4]{fonderflaass}}]
%Let $P$ be a finite poset and $J(P)$ the set of order ideals of $P$. If $P$ is the disjoint union of two posets $P=P_1+ P_2$, then $T(J(P)) = T(J(P_1))\times T(J(P_2))$. If $P$ is not the disjoint union of two posets, $T(J(P))$ is either the symmetric or alternating group on $J(P)$ (i.e.\ $J(P)$ is toggle-alternating). 
%\end{cfdf_theorem}

\begin{proof}[of Theorem~\ref{thm:fonderflaass}]
If $P$ is the disjoint union of two posets $P=P_1+ P_2$, then by Lemma~\ref{lem:oicommute}, ${J(P)}$ is a toggle-disjoint Cartesian product ${J(P_1)}\otimes {J(P_2)}$. So by Theorem~\ref{thm:times}, $T(J(P)) = T(J(P_1))\times T(J(P_2))$.

We show that the family $\{  J(P) \ | \ P \mbox{ is a connected poset}  \}$ is inductively toggle-alternating.
For each connected poset $P$, there will be either a maximal or minimal element $e\in P$ such that $P\setminus e$ is still connected, so this will suffice as the specified element in the definition of inductively toggle-alternating. (A maximal element will satisfy $(2\mathrm{b})$ in Definition~\ref{def:pleasant}, while a minimal element will satisfy $(2\mathrm{a})$.) Finally, we have checked, using \verb|SageMath|~\cite{sage}, that for all connected posets with $|P|\leq 4$, $T(J(P))=\mathfrak{S}_{J(P)}$ or ${A}_{J(P)}$, thus condition (1) of Definition~\ref{def:pleasant} is satisfied. (See Appendix~\ref{sec:appendix} for this code.) So this is an inductively toggle-alternating family, and the result follows by Theorem~\ref{thm:toggpstructure}.
\end{proof}

In~\cite{fonderflaass}, P.~Cameron and D.~Fon-der-Flaass showed that the following action is expressible in terms of the toggle group.

\begin{definition}
\label{def:row}
Given a finite poset $P$, let the \emph{rowmotion} of an order ideal $I\in J(P)$ be defined as the order ideal generated by the minimal elements of $P\setminus I$.
\end{definition}

\begin{theorem}[\protect{\cite[Lemma 1]{fonderflaass}}]
%\label{thm:linextrow}
Given any finite poset $P$, rowmotion acting on $J(P)$ is equal to the toggle 
group element given by toggling the elements of $P$ in the reverse order of any linear extension (that is, from top to bottom).
\end{theorem}

Then in~\cite{prorow}, N.~Williams and the author used the group structure of $T(J(P))$ to prove the following.
\begin{theorem}[\protect{\cite[Theorem 5.2]{prorow}}]
\label{thm:prorow}
Let $P$ be a finite ranked poset. Then there is an equivariant bijection between $J(P)$ under promotion (toggling from `left to right') and $J(P)$ under rowmotion (toggling from top to bottom).
\end{theorem}

This theorem resulted in many interesting corollaries when applied to specific posets; see Sections~6 through 8 of~\cite{prorow}.

In~\cite{DilksPechenikStriker}, K.~Dilks, O.~Pechenik, and the author introduced and developed the machinery of \emph{affine hyperplane toggles} and \emph{$n$-dimensional lattice projections}. Though it would take too long to define these here (see~\cite[Sections 3.4 and 3.5]{DilksPechenikStriker} for details), we highlight the following theorem that gives a large family of toggle group actions $\{ \pro_{\pi,v} \}$ whose orbit structures are equivalent to that of rowmotion. This is an extension from $2$ to $n$ dimensions of Theorem~\ref{thm:prorow}. 

\begin{theorem}[\protect{\cite[Theorem 3.25]{DilksPechenikStriker}}]
\label{thm:DPS}
Let $P$ be a finite poset with an $n$-dimensional lattice projection $\pi$. Let $v$ and $w$ be vectors in $\mathbb{R}^n$ with entries in $\{\pm 1\}$. 
Then there is an equivariant bijection
between $J(P)$ under $\pro_{\pi,v}$ and $J(P)$ under $\pro_{\pi,w}$. (Note $\pro_{\pi,(1,1,\ldots,1)}$ is rowmotion and $\pro_{\pi,(1,-1)}$ is the promotion of Theorem~\ref{thm:prorow} above.)
\end{theorem}

\subsection{Chains}
\label{sec:chains}
Another set with combinatorially interesting toggle group is the set of chains $\C(P)$ of a poset $P$. A \emph{chain} in a poset is a set of mutually comparable elements. In this case and several cases to follow, but unlike in the case of order ideals, any element of a chain may always be toggled `out' of the chain. That is, for any chain $C$, if $p\in C$, then $t_p(C)=C\setminus \{p\}$; since $C$ is a set of mutually comparable elements, $C\setminus \{p\}$ is as well. 

Recall from Lemma~\ref{lem:oicommute} that in the order ideal toggle group, two toggles commute if and only if there is not a covering relation between them. In contrast, we have the following toggle commutation relation for the chain toggle group.

\begin{lemma}
\label{lem:chaintog}
Let $P$ be a poset. In the chain toggle group $T(\C(P))$, $(t_p t_q)^2=1$ if and only if $p$ and $q$ are comparable in~$P$.
\end{lemma}

\begin{proof}
If $p$ and $q$ are incomparable in $P$, then $t_p t_q(\emptyset) = \{q\}$ whereas $t_q t_p(\emptyset) = \{p\}$. So $t_p$ and $t_q$ do not commute. If $p<q$, then the presence or absence of $p$ in a chain $C$ does not affect whether $q$ is in $C$, so $t_p$ and $t_q$ commute.
\end{proof}

We have the following structure theorem, as a corollary of Theorems~\ref{thm:times} and~\ref{thm:toggpstructure}. We will need the following standard poset-theoretic definition.

\begin{definition}
Let $P_1$ and $P_2$ be posets with partial orders $\leq_{P_1}$ and $\leq_{P_2}$, respectively. The \emph{ordinal sum} of $P_1$ and $P_2$, denoted $P_1\oplus P_2$, is the poset defined on the union of the underlying ground sets $P_1$ and $P_2$ with partial order $e\leq f$ if and only if one of the following holds:
\begin{itemize}
\item $e,f\in P_1$ and $e\leq_{P_1} f$, 
\item $e,f\in P_2$ and $e\leq_{P_2} f$, or
\item $e\in P_1$ and $f\in P_2$.
\end{itemize}
\end{definition}

\begin{corollary}
Let $P$ be a finite poset and $\C(P)$ the set of chains of $P$. If $P$ is the ordinal sum of two posets $P=P_1\oplus P_2$, then $T(\C(P)) = T(\C(P_1))\times T(\C(P_2))$. If $P$ is not the ordinal sum of posets, $\C(P)$ is toggle-alternating. 
\end{corollary}

\begin{proof}
Suppose $P$ is the ordinal sum $P_1\oplus P_2$; this means that in $P$, each element from $P_1$ is less than every element from $P_2$. 
Thus, a chain ${C}$ in $\C(P)$ consists of the union of any chain $C_1\in \C(P_1)$ and any chain $C_2\in \C(P_2)$, where the choice of these chains is independent. By Lemma~\ref{lem:chaintog}, $t_{p_1} t_{p_2} = t_{p_2} t_{p_1}$ for any $p_1\in P_1$, $p_2\in P_2$ since $p_1<p_2$ by the definition of ordinal sum.
Thus ${\C(P)}$ will be the toggle-disjoint Cartesian product  ${\C(P_1)}\otimes {\C(P_2)}$.
Therefore, by Theorem~\ref{thm:times}, $T(\C(P)) = T(\C(P_1))\times T(\C(P_2))$.

If $P$ is not the ordinal sum of two posets, the result follows from Theorem~\ref{thm:toggpstructure} and the fact, proved below, that the set $\left\{ \C(P) \ | \ \mbox{$P$ not an ordinal sum of two posets} \right\}$ is an inductively toggle-alternating family. 

Now suppose  $P$ is not an ordinal sum, then either $P$ is an antichain, or there exists a minimal element that is incomparable to some non-minimal element. If $P$ is an antichain, then any element may be chosen as the required $e$ in condition $(2\mathrm{b})$ of Definition~\ref{def:pleasant}. If $P$ is not an antichain, 
%so that there exists 
let $x$ be a minimal element that is incomparable to some non-minimal element $y$. Pick any other minimal element $e$ (another minimal element must exist, otherwise $P$ would be the ordinal sum of $x$ and the rest of the poset); we show this will be the required $e$ in condition $(2\mathrm{b})$ of Definition~\ref{def:pleasant}. Note that $\mathcal{L}_{\thickbar{e}}$ for $\mathcal{L}=\C(P)$ does not contain the subset $\{x,y\}$, since $\{x,y\}$ is not a chain in $P$. Since $P$ is not an ordinal sum and we are removing a minimal element $e$, the only spot where $P\setminus e$ may now be an ordinal sum must involve the minimal elements, including $x$. But $P\setminus e$ is not the ordinal sum of a subset involving $x$ and the rest of the poset, since $\{x,y\}\notin \C(P\setminus e)$.

Finally, we have checked, using \verb|SageMath|~\cite{sage} (see Appendix~\ref{sec:appendix}), that for all posets with $|P|\leq 4$ such that $P$ is not an ordinal sum, $T(\C(P))=\mathfrak{S}_{\C(P)}$ or ${A}_{\C(P)}$, thus condition $(1)$ of Definition~\ref{def:pleasant} is satisfied. So this is an inductively toggle-alternating family, and the result follows by Theorem~\ref{thm:toggpstructure}.
\end{proof}

As in the case of order ideals (see Theorems~\ref{thm:prorow} and \ref{thm:DPS} in Section~\ref{sec:orderideals}), we have a theorem which says a large class of toggle orders are equivariant.  
We will need the following lemma from \cite{humphreys}, which appears as Lemma 5.1 in \cite{prorow}.

\begin{lemma}[\protect{\cite[p.\ 74]{humphreys}}]
\label{lem:51}
Let $G$ be a group whose generators $g_1,\ldots, g_n$ satisfy $g_i^2=1$ and $(g_i g_j)^2 = 1$ if
$|i - j| > 1$. Then for any $\omega, \nu \in \mathfrak{S}_n$, $\prod_i g_{\omega(i)}$ and $\prod_i g_{\nu(i)}$ are conjugate.
\end{lemma}

For $C$ any chain of $P$, let $t_{C}$ denote the composition of all the toggles $t_p$ for  $p\in C$ (order does not matter, since by Lemma~\ref{lem:chaintog}, the toggles within a chain commute). Also, for a collection of chains $\mathbf{C}=\{C_1, C_2, \ldots, C_k\}$ and a permutation $\pi\in\mathfrak{S}_k$, let $t_{\mathbf{C}(\pi)}:=t_{C_{\pi(1)}}t_{C_{\pi(2)}}\cdots t_{C_{\pi(k)}}$.

\begin{theorem}
\label{thm:chaineq}
Fix a collection of chains $\mathbf{C}=\{C_1, C_2, \ldots, C_k\}$ in a finite poset $P$. Suppose these chains have the property that each element of $C_i$ is comparable with every element of $C_j$ provided $|i-j|>1$. Then for any $\pi, \omega\in\mathfrak{S}_k$, there is an equivariant bijection between $\mathcal{C}(P)$ under $t_{\mathbf{C}(\pi)}$ and under $t_{\mathbf{C}(\omega)}$.
\end{theorem}

\begin{proof}
For any chain $C\in P$, $t_C^2=1$, since all the toggles within a chain commute by Lemma~\ref{lem:chaintog}. Then since each element of $C_i$ is comparable with every element of $C_j$ provided $|i-j|>1$, this implies $(t_{C_i}t_{C_j})^2=1$ for $|i-j|>1$.
Therefore, the result follows from Lemma~\ref{lem:51} by considering $G$ to be the subgroup of $T(\C(P))$ generated by the toggles $\{t_{C_i} \ | \ 1\leq i\leq k\}$.
\end{proof}

\begin{example} \rm
\label{ex:chainex}
Consider the poset in Figure~\ref{fig:chainex}. Define a collection of chains as follows: $\mathbf{C}=\{C_1, C_2,\ldots, C_6\}$ where $C_i=\{i\}$. Note that each $i$ is comparable with every element in $P$ except $i\pm 1$, so $\mathbf{C}$ satisfies the conditions of Theorem~\ref{thm:chaineq}. Thus, by Theorem~\ref{thm:chaineq}, the toggle group elements obtained by toggling all the elements of this poset in any fixed order are conjugate.
\end{example}

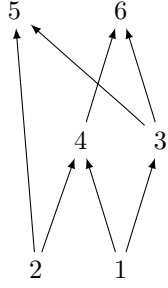
\begin{figure}[htbp]
\begin{center}
\begin{tikzpicture}[>=latex,line join=bevel,]
\node (node_5) at (46.0bp,104.5bp) [draw,draw=none] {$6$};
  \node (node_4) at (6.0bp,104.5bp) [draw,draw=none] {$5$};
  \node (node_3) at (31.0bp,55.5bp) [draw,draw=none] {$4$};
  \node (node_2) at (61.0bp,55.5bp) [draw,draw=none] {$3$};
  \node (node_1) at (14.0bp,6.5bp) [draw,draw=none] {$2$};
  \node (node_0) at (46.0bp,6.5bp) [draw,draw=none] {$1$};
  \draw [black,->] (node_0) ..controls (42.092bp,19.746bp) and (38.649bp,30.534bp)  .. (node_3);
  \draw [black,->] (node_1) ..controls (12.32bp,27.658bp) and (9.1253bp,65.997bp)  .. (node_4);
  \draw [black,->] (node_0) ..controls (49.908bp,19.746bp) and (53.351bp,30.534bp)  .. (node_2);
  \draw [black,->] (node_1) ..controls (18.429bp,19.746bp) and (22.331bp,30.534bp)  .. (node_3);
  \draw [black,->] (node_3) ..controls (34.908bp,68.746bp) and (38.351bp,79.534bp)  .. (node_5);
  \draw [black,->] (node_2) ..controls (57.092bp,68.746bp) and (53.649bp,79.534bp)  .. (node_5);
  \draw [black,->] (node_2) ..controls (46.319bp,69.045bp) and (31.3bp,81.88bp)  .. (node_4);
\end{tikzpicture}
\end{center}
\caption{The poset of Example~\ref{ex:chainex}.}
\label{fig:chainex}
\end{figure}

\subsection{Antichains} 
\label{sec:antichains}
An \emph{antichain} in a poset $P$ is a set of mutually incomparable elements; denote the set of antichains of $P$ as $\A(P)$. Though antichains are in simple bijection with order ideals (the maximal elements of an order ideal are an antichain), antichain toggles differ from order ideal toggles. With antichains, as with chains, the elements may always be toggled `out', whereas this is not the case in the order ideal toggle group.

\begin{lemma}
\label{lem:antitog}
Let $P$ be a poset. In the antichain toggle group $T(\A(P))$, $(t_p t_q)^2=1$ if and only if $p$ and $q$ are incomparable in $P$.
\end{lemma}

\begin{proof}
Note $\emptyset$ and single element subsets are all in $\A(P)$.
If $p<q$ in $P$, then $t_p t_q(\emptyset) = \{q\}$ whereas $t_q t_p(\emptyset) = \{p\}$. So $t_p$ and $t_q$ do not commute. If $p$ and $q$ are incomparable, then the presence or absence of $p$ in an antichain $A$ does not affect whether $q$ is in $A$. So $t_p$ and $t_q$ commute.
\end{proof}

This differs from the case of order ideals in which $t_p$ and $t_q$ are noncommuting only when $p$ and $q$ share a covering relation. %whereas here $t_p$ and $t_q$ are noncommuting whenever $p$ and $q$ are comparable.
But as in the case of order ideals, we have the following structure description as a corollary of Theorems~\ref{thm:times} and~\ref{thm:toggpstructure}.
\begin{corollary}
Let $P$ be a finite poset and $\A(P)$ the set of antichains of $P$. If $P$ is the disjoint union of two posets $P=P_1+ P_2$, then $T(\A(P)) = T(\A(P_1))\times T(\A(P_2))$. If $P$ is not the disjoint union of two posets, $\A(P)$ is toggle-alternating. 
\end{corollary}

\begin{proof}
If $P=P_1 + P_2$, an antichain $A$ in $\A(P)$ consists of the union of any antichain $A_1$ in $\A(P_1)$ and any antichain $A_2$ in $\A(P_2)$, where the choice of $A_1$ and $A_2$ is independent. Now since $p_1$ and $p_2$ are incomparable for every $p_1\in P_1$, $p_2\in P_2$, we have by Lemma~\ref{lem:antitog} $t_{p_1} t_{p_2} = t_{p_2} t_{p_1}$.
Thus, ${\A(P)}$ will be the toggle-disjoint Cartesian product ${\A(P_1)} \otimes {\A(P_2)}$. 
Therefore, by Theorem~\ref{thm:times}, $T(\A(P)) = T(\A(P_1))\times T(\A(P_2))$.

If $P$ is not the disjoint union of two posets, the result follows from Theorem~\ref{thm:toggpstructure} and the fact, proved below, that the set $\{ \mathcal{A}(P) \ | \ P \mbox{ is a connected poset}  \}$ is an inductively toggle-alternating family. 

For given a connected poset $P$, you may choose any maximal or minimal element $e\in P$ such that the poset $P\setminus{\{e\}}$ is also connected to be the required $e$ in condition $(2\mathrm{b})$ of Definition~\ref{def:pleasant}. 

Finally, we have checked, using \verb|SageMath|~\cite{sage} (see Appendix~\ref{sec:appendix}), that for all connected posets with $|P|\leq 4$, $T(\A(P))=\mathfrak{S}_{\A(P)}$ or ${A}_{\A(P)}$, thus condition $(1)$ of Definition~\ref{def:pleasant} is satisfied. So this is an inductively toggle-alternating family, and the result follows by Theorem~\ref{thm:toggpstructure}.
\end{proof}

As in the cases of order ideals and chains, we have a theorem which says a large class of toggle orders are equivariant. 
For $A$ any antichain of $P$, let $t_{A}$ denote the composition of all the toggles $t_p$ for  $p\in A$ (order does not matter, since by Lemma~\ref{lem:antitog}, the toggles within an antichain commute). Also, for a collection of antichains $\mathbf{A}=\{A_1, A_2, \ldots, A_k\}$ and a permutation $\pi\in\mathfrak{S}_k$, let $t_{\mathbf{A}(\pi)}:=t_{A_{\pi(1)}}t_{A_{\pi(2)}}\cdots t_{A_{\pi(k)}}$.

\begin{theorem}
\label{thm:antichaineq}
Fix a collection of antichains $\mathbf{A}=\{A_1, A_2, \ldots, A_k\}$ in a finite poset $P$. Suppose these antichains have the property that each element of $A_i$ is incomparable with every element of $A_j$ provided $|i-j|>1$. Then for any $\pi, \omega\in\mathfrak{S}_k$, there is an equivariant bijection between $\mathcal{A}(P)$ under $t_{\mathbf{A}(\pi)}$ and under $t_{\mathbf{A}(\omega)}$.
\end{theorem}

\begin{proof}
For any antichain $A\in P$, $t_A^2=1$, since all the toggles within an antichain commute by Lemma~\ref{lem:antitog}. Then since each element of $A_i$ is incomparable with every element of $A_j$ provided $|i-j|>1$, this implies $(t_{A_i}t_{A_j})^2=1$ for $|i-j|>1$.
Therefore, the result follows from Lemma~\ref{lem:51} by considering $G$ to be the subgroup of $T(\A(P))$ generated by the toggles $\{t_{A_i} \ | \  1\leq i\leq k\}$.
\end{proof}

\begin{example} \rm
\label{ex:antichainex}
Let $P$ be the poset in Figure~\ref{fig:antichainex2}. Define a collection of antichains as follows: $\mathbf{A}=\{A_1,\ldots, A_6\}$ where $A_i=\{i\}$. Note that each $i$ is incomparable with every element in $P$ except $i\pm 1$, so $\mathbf{A}$ satisfies the conditions of Theorem~\ref{thm:antichaineq}. Thus, by Theorem~\ref{thm:antichaineq}, the toggle group elements obtained by toggling all the elements of this poset in any fixed order are conjugate.
\end{example}

\begin{figure}[htbp]
\begin{center}
\begin{tikzpicture}[>=latex,line join=bevel,]
\node (node_5) at (66.0bp,55.5bp) [draw,draw=none] {$6$};
  \node (node_4) at (66.0bp,6.5bp) [draw,draw=none] {$5$};
  \node (node_3) at (36.0bp,55.5bp) [draw,draw=none] {$4$};
  \node (node_2) at (36.0bp,6.5bp) [draw,draw=none] {$3$};
  \node (node_1) at (6.0bp,55.5bp) [draw,draw=none] {$2$};
  \node (node_0) at (6.0bp,6.5bp) [draw,draw=none] {$1$};
  \draw [black,->] (node_2) ..controls (28.047bp,19.96bp) and (20.84bp,31.25bp)  .. (node_1);
  \draw [black,->] (node_0) ..controls (6.0bp,19.603bp) and (6.0bp,30.062bp)  .. (node_1);
  \draw [black,->] (node_4) ..controls (58.047bp,19.96bp) and (50.84bp,31.25bp)  .. (node_3);
  \draw [black,->] (node_2) ..controls (36.0bp,19.603bp) and (36.0bp,30.062bp)  .. (node_3);
  \draw [black,->] (node_4) ..controls (66.0bp,19.603bp) and (66.0bp,30.062bp)  .. (node_5);
\end{tikzpicture}
\end{center}
\caption{The poset of Example~\ref{ex:antichainex}.}
\label{fig:antichainex2}
\end{figure}
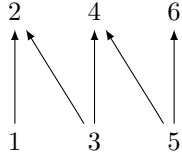

\begin{remark} \rm
Examples~\ref{ex:chainex} and \ref{ex:antichainex} are the same generalized toggle group in different settings, since the set of chains in Example~\ref{ex:chainex} (as a family of subsets of ${\{1,2,\ldots,6\}}$) equals the set of antichains in~\ref{ex:antichainex}.
\end{remark}

\subsection{Interval-closed sets}
\label{sec:ic}
An \emph{interval-closed set} in a poset $P$ is a subset $I\subseteq P$ such that if $x,y\in I$ and $x\leq y$, then for any $z\in P$ satisfying $x\leq z\leq y$ we also have $z\in I$. Denote the set of interval-closed sets of a poset as $\mathcal{IC}(P)$. 

Though the toggle group of interval-closed sets behaves similarly to the order ideal toggle group, we will see in Section~\ref{sec:row} that \emph{cover-closure} on $\mathcal{IC}(P)$ behaves quite differently than in the case of order ideals. In particular, we will not generally be able to express cover-closure as a product of toggles as in the case of rowmotion on order ideals.

\smallskip
We have the following toggle commutativity relation.
\begin{lemma}
\label{lem:ic}
In the interval-closed set toggle group $T(\mathcal{IC}(P))$, $(t_p t_q)^2=1$ if and only if $p$ and $q$ are incomparable in $P$ or (in the case $p<q$) if $q$ covers $p$ where $q$ is maximal and $p$ is minimal.
\end{lemma}

\begin{proof}
Note $\emptyset\in \mathcal{IC}(P)$. If $p<q$ in $P$ and $q$ does not cover $p$, then $t_p t_q(\emptyset) = \{q\}$ whereas $t_q t_p(\emptyset) = \{p\}$. So $t_p$ and $t_q$ do not commute. 

If $q$ covers $p$ 
and there exists an element $r$ covered by $p$, then $t_p t_q(\{r\}) = \{r\}$ whereas $t_q t_p(\{r\}) = \{p,q,r\}$. If $q$ covers $p$ 
and there exists an element $s$ covering $q$, $t_p t_q(\{s\}) = \{p,q,s\}$ whereas $t_q t_p(\{s\}) = \{s\}$. If no such $r$ or $s$ exists, that is, if $p$ is minimal and $q$ is maximal, then $t_p$ and $t_q$ commute. This is because the union of any interval-closed set with any subset of $\{p,q\}$ is still an interval-closed set, so $t_p$ and $t_q$ always act nontrivially on any interval-closed set.

If $p$ and $q$ are incomparable, then by definition the presence or absence of $p$ in an interval-closed set $I$ does not affect whether $q$ is in $I$. So $t_p$ and $t_q$ commute.
\end{proof}

We have the following structure description in Corollary~\ref{cor:icsets} as a corollary of Theorems~\ref{thm:times}
 and~\ref{thm:toggpstructure} and Lemma~\ref{lem:ic}. First, we will need the following non-standard definitions.

\begin{definition}
Let $P$ be a poset. 
If an element $m\in P$ is a maximal element that only covers minimal elements or a minimal element that is only covered by maximal elements, call $m$ \emph{extremal-atomic}. 
\end{definition}

\begin{definition}
\label{def:seaf}
Given a poset $P$ and an ordered subset of elements $p_1,p_2,\ldots,$ $p_n$ in  $P$, 
define $P_i$ recursively as the induced subposet $P_i:=P_{i-1}\setminus p_i$ with $P_0:=P$.
If $P_n$ is a chain with at least three elements,  $p_{i+1}$ is maximal or minimal in $P_i$ for each $0\leq i\leq n-1$, and each $P_i$ has connected Hasse diagram and no extremal-atomic element, call $P$ \emph{strongly-extremal-atomic-free}.
\end{definition}

\begin{corollary}
\label{cor:icsets} Let $P$ be a finite poset and $\I(P)$ the set of interval-closed sets of $P$. 
If $P$ is the disjoint union of two posets $P=P_1+ P_2$, then $T(\I(P)) = T(\I(P_1))\times T(\I(P_2))$. Suppose $P$ is not the disjoint union of two posets. If $P$ contains an extremal-atomic element $m$,
then $T(\I(P))=T(\I(P\setminus \{m\}))\times T(\I(\{m\}))= T(\I(P\setminus \{m\}))\times \mathfrak{S}_2$. (So, in particular, if $P$ has only two ranks, then $T(P)=(\mathfrak{S}_2)^{|P|}$.)
If $P$ is strongly-extremal-atomic-free, then
$\I(P)$ is toggle-alternating. 
\end{corollary}

\begin{proof}
If $P=P_1 + P_2$, an interval-closed set $I$ in $\I(P)$ consists of the union of any interval-closed set $I_1$ in $\I(P_1)$ and any interval-closed set $I_2$ in $\I(P_2)$, where the choice of $I_1$ and $I_2$ is independent. Now since $p_1$ and $p_2$ are incomparable for every $p_1\in P_1$, $p_2\in P_2$, we have by Lemma~\ref{lem:ic}, $t_{p_1} t_{p_2} = t_{p_2} t_{p_1}$. Thus, $\I(P)$ is a toggle-disjoint Cartesian product of $P_1$ and $P_2$.
Therefore by Theorem~\ref{thm:times}, $T(\I(P)) = T(\I(P_1))\times T(\I(P_2))$.

If $P$ contains an 
extremal-atomic element $m$, then by Lemma~\ref{lem:ic}, $m$ commutes with all other elements of $P$. So $T(\I(P))=T(\I(P\setminus \{m\}))\times T(\I(\{m\}))= T(\I(P\setminus \{m\}))\times \mathfrak{S}_2$.

Suppose $P$ is strongly-extremal-atomic-free. In this case, the result follows from Theorem~\ref{thm:toggpstructure} if we can show that the set of interval-closed sets of strongly-extremal-atomic-free posets
is an inductively toggle-alternating family. 
Given a strongly-extremal-atomic-free poset with its specialized sequence of elements $\{p_1,p_2,\ldots,p_n\}$, 
$p_1$ will be the required $e$ in condition $(2\mathrm{b})$ of Definition~\ref{def:pleasant}. 
We have checked, using \verb|SageMath|~\cite{sage} (see Appendix~\ref{sec:appendix}), that for all strongly-extremal-atomic-free posets with $|P|\leq 4$, $T(\I(P))=\mathfrak{S}_{\I(P)}$ or ${A}_{\I(P)}$, thus condition $(1)$ in Definition~\ref{def:pleasant} is satisfied. 
So this is an inductively toggle-alternating family, and the result follows by Theorem~\ref{thm:toggpstructure}.
\end{proof}

\begin{remark} \rm
\label{remark:icnotpleasant}
We speculate that in Corollary~\ref{cor:icsets}, the condition 
`strongly-extremal-atomic-free' may be relaxed to `contains no extremal-atomic 
element'; we have checked using \verb|SageMath|~\cite{sage} (see Appendix~\ref{sec:appendix}) that this holds for $|P|\leq 4$. A proof in the case of this relaxed condition would then complete the toggle group structure theorem for interval-closed sets on any poset. For example, the poset $P$ defined by the covers: $a\lessdot b\lessdot c, a\lessdot d\lessdot e$ contains no extremal-atomic element, but is not strongly-extremal-atomic-free. And yet, we have computed that $|\I(P)|=25$ and $T(\I(P))=\mathfrak{S}_{25}$, which agrees with our speculation. 
Proving the interval-closed set toggle group structure description for all posets satisfying this relaxed condition would take some work beyond the statement of Theorem~\ref{thm:toggpstructure}; note that no element of our example poset $P$ satisfies condition $(2)$ in Definition~\ref{def:pleasant} (and there is an infinite family of such posets). Thus, this is an example of a family of subsets that is toggle-alternating, since it contains the alternating group, but not inductively toggle-alternating.
\end{remark}

\subsection{Unions of subset families}
\label{sec:multipleposets}
Let $\mathcal{L}=\mathcal{L}_1\cup\mathcal{L}_2$ be the union of two (not necessarily disjoint) combinatorially meaningful sets and consider the toggle group on this union. For example, fix two posets $P$ and $P'$ on the same underlying set. We may consider the toggle group $T(\mathcal{L})$ for $\mathcal{L}=J(P)\cup J(P')$, the union of the two sets of order ideals. For $p\in P$ and $I\in\mathcal{L}$,  the toggle $t_p$ is the symmetric difference of $p$ and $I$, provided this is an order ideal in \emph{either} $P$ or $P'$. As another example, 
it may be interesting to study the toggle group $T(\mathcal{L})$ where $\mathcal{L}$ is the union of order ideals and filters, or the union of order ideals and antichains, or the union of antichains under two different partial orders, etc.

%We have the following (partial) commutation lemma, whose proof is immediate. 

%\begin{lemma}
%Given two partial orders $\leq$ and $\leq'$ on the same ground set $P$, consider the toggle group on the union of their sets of order ideals $T(J(P,\leq)\cup J(P,\leq'))$. In this toggle group, we have $(t_p t_q)^2=1$ if $p$ and $q$ are incomparable in both $(P,\leq)$ and $(P,\leq')$.
%\end{lemma}

\begin{remark} \rm
\label{remark:bij}
Suppose $|\mathcal{L}_1| = |\mathcal{L}_2|$.  An explicit bijection between $\mathcal{L}_1$ and $\mathcal{L}_2$ could be given by finding an appropriate toggle group element in $T(\mathcal{L}_1\cup \mathcal{L}_2)$ or $T(\mathcal{L})$ for $\mathcal{L}\supset \mathcal{L}_1\cup \mathcal{L}_2$. Such an element will always exist if the toggle group is the full symmetric group, as is often true. This may be a method for finding missing bijections. An example is the open problem of finding an explicit bijection between $n\times n$ alternating sign matrices and totally symmetric self-complementary plane partitions in a $2n\times 2n\times 2n$ box, both of which are known to be counted by $\prod_{j=0}^{n-1}\frac{(3n+1)!}{(n+j)!}$. In~\cite{StrikerPoset}, we exhibited these sets as order ideals in two different partial orders on the same ground set. Thus, if one could describe, for every $n$, a toggle group element on the union of these sets of order ideals that takes one set to the other, this would give a bijection. 
\end{remark}

\subsection{Independent sets of a graph}
\label{sec:indepsets} 

We now move from the realm of posets to that of graphs. An \emph{independent set} in a graph is a subset of the vertices such that no two have an edge between them. Let $\IS(G)$ equal the set of independent sets of a graph $G = (V(G),E(G))$, where $V(G)$ and $E(G)$ are the vertex and edge sets, respectively.  Since independent sets are subsets of the vertices, we take $V(G)$ as our ground set and $\mathcal{L}=\IS(G)\subseteq 2^{V(G)}$.

With independent sets, as with chains and antichains, the elements may always be toggled `out' of an independent set. We have the following toggle relation.
\begin{lemma}
\label{lem:is}
In the independent set toggle group $T(\IS(G))$, $(t_u t_v)^2=1$ if and only if there is no edge between $u$ and $v$ in $G$.
\end{lemma}

\begin{proof}
If there is an edge between $u$ and $v$ in $G$, then $t_u t_v(\emptyset) = \{v\}$ whereas $t_v t_u(\emptyset) = \{u\}$. So $t_u$ and $t_v$ do not commute. If there is no edge between $u$ and $v$, then the presence or absence of $u$ in an independent set $I$ does not affect whether $v$ is in $I$. So $t_u$ and $t_v$ commute.
\end{proof}

We have the following structure description as a corollary of Theorems~\ref{thm:times} and~\ref{thm:toggpstructure}.
\begin{corollary}
\label{cor:is}
Let $G$ be a finite graph and $\IS(G)$ the set of independent sets of $G$. If $G$ is connected, $\IS(G)$ is toggle-alternating. If $G$ has components $G_1,G_2,\ldots, G_k$, then $T(\IS(G)) = T(\IS(G_1))\times T(\IS(G_2))\times\cdots\times T(\IS(G_k))$.
\end{corollary}

\begin{proof}
If $G$ is disconnected with components $G_1,G_2,\ldots, G_k$, an independent set $X$ in $\IS(G)$ consists of the disjoint union of the independent sets $X_i$ in $\IS(G_i)$, $1\leq i\leq k$, where the choice of the $X_i$ is independent. Now since there is no edge between any of $g_1,g_2,\ldots,g_k$ for $g_i\in G_i$,  $1\leq i\leq k$, we have by Lemma~\ref{lem:is}, $t_{g_i}$ and $t_{g_j}$ commute for all $1\leq i,j\leq k$, so that ${\IS(G)}$ is the toggle-disjoint Cartesian product ${\IS(G_1)}\otimes{\IS(G_2)}\otimes\cdots\otimes{\IS(G_k)}$. 
Therefore, by Theorem~\ref{thm:times}, $T(\IS(G)) = T(\IS(G_1))\times T(\IS(G_2))\times\cdots\times T(\IS(G_k))$.

If $G$ is a connected graph, the result follows from Theorem~\ref{thm:toggpstructure} and the fact, proved below, that the set $\{  \IS(G) \  | \  G \mbox{ is a connected graph}\}$ is an inductively toggle-alternating family. 

For given a connected graph $G$, one may choose any vertex that when deleted does not disconnect the graph  
 to be the required $e$ in condition $(2\mathrm{b})$ of Definition~\ref{def:pleasant}. 

Finally, we have checked, using \verb|SageMath|~\cite{sage} (see Appendix~\ref{sec:appendix}), that for all connected graphs with $|G|\leq 4$, $T(\IS(P))=\mathfrak{S}_{\IS(P)}$ or ${A}_{\IS(P)}$, thus condition $(1)$ of Definition~\ref{def:pleasant} is satisfied. So this is an inductively toggle-alternating family, and the result follows by Theorem~\ref{thm:toggpstructure}.
\end{proof}
 
\begin{remark} \rm
\label{remark:indep}
Toggling independent sets is related to Monte-Carlo simulations and Glauber dynamics of the hard-core model of statistical physics~\cite[Chapter 3]{markovmixing}.
Also, recently, D.~Einstein, M.~Farber,  E.~Gunawan,  M.~Joseph,  M.~Macauley,  J.~Propp, and S.~Rubinstein-Salzedo have used this perspective of generalized toggling to prove some results on toggling noncrossing matchings and, more generally, independent sets of certain graphs~\cite{EFGJMPRS}. Finally, M.~Joseph and T.~Roby have recently found orders of some toggle group elements and other related results in the toggle group of independent sets of a path graph~\cite{MJoseph}.
\end{remark}

\subsection{Vertex covers of a graph}
\label{sec:vertexcovers}
 A \emph{vertex cover} of a graph $G$ is a subset of the vertices $V(G)$ such that every edge in $E(G)$ is incident to at least one vertex in the vertex cover. Denote the set of vertex covers of $G$ as $\mathcal{VC}(G)$. Since vertex covers are subsets of the vertices, we take $V(G)$ as our ground set and $\mathcal{L}=\mathcal{VC}(G)\subseteq 2^{V(G)}$. In the vertex cover toggle group $T(\mathcal{VC}(G))$, you may always toggle a vertex `in' to a vertex cover, but you may not always toggle a vertex `out'. 

\smallskip
We have the following toggle commutation lemma.
\begin{lemma}
\label{lem:vc}
Let $u$ and $v$ be vertices of $G$. In the vertex cover toggle group $T(\mathcal{VC}(G))$, $(t_u t_v)^2=1$ if and only if there is no edge between $u$ and $v$.
\end{lemma}

\begin{proof}
Suppose there is an edge between $u$ and $v$. Then consider the vertex cover $X$ consisting of all the vertices in $V(G)$. We have that $t_ut_v(X)=X\setminus\{v\}$ while $t_vt_u(X)=X\setminus\{u\}$. So $t_u$ and $t_v$ do not commute.

If $u$ and $v$ do not have an edge between them, then the presence or absence of $u$ in a vertex cover has no affect on whether $v$ must be in the vertex cover. So $t_u$ and $t_v$ commute.
\end{proof}

We have the following structure description. 
The proof is similar to the proof of Corollary~\ref{cor:is}, since the commutativity lemma is the same; the main difference is that condition $(2\mathrm{a})$ of Definition~\ref{def:pleasant} is used rather than condition $(2\mathrm{b})$. We have checked the base cases $|G|\leq 4$ using \verb|SageMath|~\cite{sage} (see Appendix~\ref{sec:appendix}). 

\begin{corollary}
Let $G$ be a finite graph and $\mathcal{VC}(G)$ the set of vertex covers of $G$. If $G$ is connected, $T(\mathcal{VC}(G))$ is either the symmetric group or the alternating group on $\mathcal{VC}(G)$. If $G$ has components $G_1,G_2,\ldots, G_k$, then $T(\mathcal{VC}(G)) = T(\mathcal{VC}(G_1))\times T(\mathcal{VC}(G_2))\times\cdots\times T(\mathcal{VC}(G_k))$.
\end{corollary}

\begin{remark} \rm
\label{remark:vc}
Vertex covers have been used in commutative algebra work of B.~Kubik, C.~Paulsen, and S.~Sather-Wagstaff to index the decomposition of edge ideals~\cite{paulsen} and path ideals~\cite{kubik}. In this work, minimal vertex covers are of special importance. We can characterize a minimal vertex cover $X$ as a vertex cover which is invariant under any of the toggles $t_v$ for $v\in X$. It would be interesting to see what algebraic implications the action of the toggle group on vertex covers may have. 
\end{remark}

\subsection{Acyclic subgraphs of a graph} 
\label{sec:acyclic}
We consider an \emph{acyclic subgraph} of a graph $G$ to be a collection of edges in $E(G)$ in which there is no cycle. Denote the set of acyclic subgraphs of $G$ as $\mathcal{AS}(G)$. Since acyclic subgraphs are subsets of the edges rather than the vertices, we take $E(G)$ as our ground set and $\mathcal{L}=\mathcal{AS}(G)\subseteq 2^{E(G)}$.

The maximal acyclic subgraphs are the spanning forests. In the acyclic subgraph toggle group $T(\mathcal{AS}(G))$, you may always toggle an edge `out', but you may not always toggle an edge `in', since adding an edge may introduce a cycle.

We have the following commutation lemma.
\begin{lemma}
\label{lem:as}
Let $e$ and $f$ be edges of $G$. In the acyclic subgraph toggle group $T(\mathcal{AS}(G))$, $(t_e t_f)^2=1$ if and only if no cycle of $G$ contains both $e$ and $f$.
\end{lemma}

\begin{proof}
Suppose a cycle $C$ of $G$ contains edges $e$ and $f$; consider the acyclic subgraph $C\setminus e$. We know $t_e t_f(C\setminus e)= C\setminus f$ since both toggles act nontrivially whereas $t_f t_e(C\setminus e)= C\setminus \{e,f\}$ since here $t_e$ must act as the identity (or else the resulting subgraph would be a cycle). So $t_e$ and $t_f$ do not commute. 

If no cycle contains both $e$ and $f$, then the presence or absence of $e$ in a subgraph does not affect whether $f$ may be in the subgraph while maintaining acyclicity, and vice versa. So in this case $t_e$ and $t_f$ commute.
\end{proof}

We have the following partial structure description as a corollary of Theorem~\ref{thm:times}. Note a \emph{cut-vertex} of a graph is a vertex that disconnects the graph when removed from the graph.
\begin{corollary}
\label{cor:as}
Let $G$ be a finite graph and $\mathcal{AS}(G)$ the set of acyclic subgraphs of $G$. 
Suppose $G$ has an edge $e$ not contained in any cycle of $G$. Then $T(\mathcal{AS}(G)) =  T(\mathcal{AS}(G\setminus \{e\}))\times T(\mathcal{AS}(\{e\})) = T(\mathcal{AS}(G\setminus \{e\}))\times \mathfrak{S}_2$. 
Let $G_1,G_2,\ldots, G_k$ be subgraphs of $G$ with pairwise disjoint edge sets $E_1,E_2,\ldots,E_k$ such that $E=\cup_{i=1}^k E_i$ and such that any pair of subgraphs $G_i$ and $G_j$ have at most one vertex in common, and this vertex is a cut-vertex of $G$. Then $T(\mathcal{AS}(G)) = T(\mathcal{AS}(G_1))\times T(\mathcal{AS}(G_2))\times\cdots\times T(\mathcal{AS}(G_k))$. 
\end{corollary}

\begin{proof}
Suppose $G_1,G_2,\ldots, G_k$ are subgraphs of $G$ 
with edge sets $E_1,E_2,\ldots,E_k$ satisfying the hypotheses.
%such that $E=\cup_{i=1}^k E_i$ and such that any pair of subgraphs $G_i$ and $G_j$ have at most one vertex in common, and this vertex is a cut-vertex of $G$. 
An acyclic subgraph $X$ in $\mathcal{AS}(G)$ consists of the disjoint union of the acyclic subgraphs $X_i\subseteq E_i$ in $\mathcal{AS}(G_i)$, $1\leq i\leq k$, where the choice of the $X_i$ is independent. Now since no edges from different components can be in the same cycle, we have by Lemma~\ref{lem:as}, $t_{e_i}$ and $t_{e_j}$ commute for all $e_i\in E_i$, $e_j\in E_j$, $1\leq i,j\leq k$ so that ${\mathcal{AS}(G)}$ is the toggle-disjoint Cartesian product ${\IS(G_1)}\otimes{\IS(G_2)}\otimes\cdots\otimes{\IS(G_k)}$. 
Therefore, by Theorem~\ref{thm:times}, $T(\IS(G)) = T(\IS(G_1))\times T(\IS(G_2))\times\cdots\times T(\IS(G_k))$.

Suppose $G$ has an edge $e$ not contained in any cycle of $G$. Then by Lemma~\ref{lem:as}, $t_e$ commutes with all other toggles in $T(\mathcal{AS}(G))$. So $T(\mathcal{AS}(G)) =  T(\mathcal{AS}(G\setminus \{e\}))\times T(\mathcal{AS}(\{e\})) = T(\mathcal{AS}(G\setminus \{e\}))\times \mathfrak{S}_2$.
\end{proof}

\begin{remark} \rm
Note  $\{  \mathcal{AS}(G) \  | \  G$  is a connected graph in which each edge is contained in at least one cycle$\}$ is not an inductively toggle-alternating family. Take, for example, the cycle $C_n$. Label the edges with $[n]=\{1,2,\ldots,n\}$. $\mathcal{AS}(C_n)=2^{[n]}\setminus \{[n]\}$, since the only cycle is $C_n$ itself.
But if we remove any edge, the graph becomes a path. Since there are no cycles in a path, all toggles commute, so the toggle group of a path with $n$ edges is $(\mathfrak{S}_2)^n$. So there is no edge satisfying Condition $(2)$ of Definition~\ref{def:pleasant}. 
\end{remark}

\subsection{Spanning subgraphs of a graph}
\label{sec:conn}
A \emph{spanning subgraph} of a graph $G$ is a subset $S\subseteq E(G)$ of the edges of $G$
such that the subgraph $(V(G), S)$ consisting of all the vertices of $G$ and the edges
in $S$ has the same number of connected components as the original graph $G$.
We take as our ground set the set of edges $E(G)$, and our subsets $\mathcal{L}=\mathcal{S}(G)\subseteq 2^{E(G)}$ to be the edge sets of spanning subgraphs. We define the spanning subgraph toggle group as the subgroup of the symmetric group on all spanning subgraphs generated by the edge toggles.

The minimal spanning subgraphs are the {spanning forests}. In the spanning subgraph toggle group $T(\mathcal{S}(G))$, you may always toggle an edge `in', but you may not always toggle an edge `out', since removing an edge may increase the number of components. These observations indicate the spanning subgraph toggle group is in some sense `dual' to the acyclic subgraph toggle group; we will return to this in the next subsection.

We give below a commutation lemma, in which we use the following standard definition.
\begin{definition}
 A \emph{cut} of a graph is a partition of the vertices into two disjoint subsets. A cut determines a \emph{cut-set} (also called a \emph{co-circuit}), defined as the set of edges that have one endpoint in each subset of the partition.
\end{definition} 
 
\begin{lemma}
\label{lem:ss}
Let $e$ and $f$ be edges of $G$. In the spanning subgraph toggle group $T(\mathcal{S}(G))$, $(t_e t_f)^2=1$ if and only if no cut-set  of $G$ contains both $e$ and $f$.
\end{lemma}

\begin{proof}
Suppose a cut-set $C$ of $G$ contains edges $e$ and $f$; consider the spanning subgraph $(E(G)\setminus C) \cup \{f\}$. We know $t_f t_e((E(G)\setminus C) \cup \{f\})= (E(G)\setminus C) \cup \{e\}$ since both toggles act nontrivially whereas $t_e t_f((E(G)\setminus C) \cup \{f\})= (E(G)\setminus C) \cup \{e,f\})$ since here $t_f$ must act as the identity (or else the resulting subgraph would no longer be spanning). So $t_e$ and $t_f$ do not commute. 

If no cut-set contains both $e$ and $f$, then the presence or absence of $e$ in a subgraph does not affect whether $f$ may be in the subgraph while still spanning, and vice versa. So in this case $t_e$ and $t_f$ commute.
\end{proof}

\subsection{Matroids}
\label{sec:matroids} 
A matroid is an abstraction of the notion of linear independence. 
Matroids can be defined in many ways; we give two definitions here and characterize part of each definition in terms of toggles. 

\begin{definition}
A finite \emph{matroid} $M$ may be defined as a pair $(E,\mathcal{I})$, where $E$ is a finite set (called the ground set) and $\mathcal{I}$ is a set of subsets of $E$ (called the independent sets) with the following properties:

\begin{enumerate}
\item
$\emptyset\in\mathcal{I}$;
\item
For each $X\subseteq Y\subseteq E$, if $Y\in\mathcal{I}$ then $X\in\mathcal{I}$ (the \emph{hereditary} property); and
\item
If $X,Y\in \mathcal{I}$ and $|Y|>|X|$, then there exists $y\in Y\setminus X$ such that $X\cup \{y\}\in\mathcal{I}$ (the \emph{independent set exchange} property).
\end{enumerate}

A maximal independent subset is a \emph{basis}. A subset of $E$ is \emph{dependent} if it is not independent (that is, not in $\mathcal{I}$). A minimal dependent subset of $E$ is a \emph{circuit}. 
\end{definition}

Since the set $\mathcal{I}$ of independent sets is a set of subsets of $E$, we may define and study the toggle group $T(\mathcal{I})$. For any $X\in \mathcal{I}$ and any $x\in X$, $t_x(X)=X\setminus \{x\}$. That is, we may always toggle elements `out' of an independent set in the toggle group of independent sets of a matroid.

\begin{remark} \rm
\label{remark:matroid1}
We can characterize the independent set exchange condition above using the toggle group $T(\mathcal{I})$. Namely, (3) is equivalent to the following:
\begin{enumerate}
\item[3*.]
If $X,Y\in \mathcal{I}$ and $|Y|>|X|$, then there exists $y\in Y\setminus X$ such that $t_y(X)\neq X$ (that is, $t_y$ acts nontrivially on~$X$ in $T(\mathcal{I})$).
\end{enumerate}
\end{remark}

We note the toggle groups considered in the previous two subsections are examples of toggling independent sets in matroids, namely, graphic and co-graphic matroids. We give here a general commutation lemma for matroid toggle groups that specializes to Lemmas~\ref{lem:as} and \ref{lem:ss}.

\begin{lemma}
Let $M=(E,\mathcal{I})$ be a matroid and $x,y\in E$. In the matroid independent set toggle group $T(\mathcal{I})$, $(t_x t_y)^2=1$ if and only if no circuit of $M$ contains both $x$ and $y$.
\end{lemma}

\begin{proof} 
Suppose a circuit $C$ of $M$ contains both $x$ and $y$. Since $C$ is a minimal dependent set, $C\setminus\{x\},C\setminus \{y\}$, and $C\setminus \{x,y\}$ are all in $\mathcal{I}$. We know $t_x t_y(C\setminus \{x\})= C\setminus \{y\}$ since both toggles act nontrivially, whereas $t_y t_x(C\setminus \{x\})= t_y (C\setminus \{x\})= C\setminus \{x,y\}$ since here $t_x$ must act as the identity (or else the resulting subset would no longer be independent). So $t_x$ and $t_y$ do not commute. 

If no circuit contains both $x$ and $y$, then the presence or absence of $x$ in a independent set does not affect whether $y$ may be in the independent set, and vice versa. So in this case $t_x$ and $t_y$ commute.
\end{proof}

\subsection{Convex geometries}
\label{sec:antimatroids}
Just as matroids are an abstraction of the notion of linear independence, convex geometries are an abstraction of the notion of convexity. Convex geometries are dual to \emph{antimatroids} and also equivalent to \emph{meet-distributive lattices}; see~\cite{EdelmanJamison}.  As with matroids, we first state the definition and then characterize part of the definition in terms of toggles. We will also refer to Definition~\ref{def:cxgeom} in Section~\ref{sec:row}.
\begin{definition}
\label{def:cxgeom}
Let $E$ be a finite nonempty set. A set $\mathcal{L}$ of subsets of $E$ is a \emph{convex geometry}  on $E$ if $\mathcal{L}$ satisfies the following properties.
\begin{enumerate}
\item $\emptyset\in\mathcal{L}$ and $E\in\mathcal{L}$;
\item If $X,Y\in\mathcal{L}$, then $X\cap Y\in\mathcal{L}$; and
\item If $X\in\mathcal{L}\setminus\{E\}$, then there exists $e\in E\setminus X$ such that $X\cup\{e\}\in\mathcal{L}$.
\end{enumerate}
The elements of $\mathcal{L}$ are called the \emph{convex sets}.
\end{definition}

\begin{remark} \rm
\label{remark:cxgeom}
We can recast condition $(3)$ above in terms of toggles:
\begin{enumerate}
\item[3*.]
For any $X\in\mathcal{L}\setminus\{E\}$, there exists a toggle $t_e$ with $e\in E\setminus X$ such that $t_e(X)\neq X$.
\end{enumerate}
\end{remark}

\begin{example} \rm
\label{ex:notpleasant}
Let $E=\left\{1,2,3\right\}$ and 
$\mathcal{L}=\left\{\emptyset, \{1\}, \{2\}, \{3\}, \{1,2\}, \{2,3\}, \{1,2,3\}\right\}$. $\mathcal{L}$ is a convex geometry, but note that $\mathcal{L}$ is not closed under set union since $\{1,3\}\notin\mathcal{L}$. So, in particular, $\mathcal{P}_{\mathcal{L}}$ is not a distributive lattice, so there is no poset $P$ such that $\mathcal{L}=J(P)$.
\end{example}

Convex geometries have a closure operator, defined below. We study more general closure operators in the next section.
\begin{definition}
For $A\subseteq E$ and $\mathcal{L}\subseteq 2^E$ a convex geometry, define the \emph{closure operator} of $\mathcal{L}$, $\tau_{\mathcal{L}}:2^E\rightarrow\mathcal{L}$, as
\[\tau_{\mathcal{L}}(A)=\bigcap\left\{X\in\mathcal{L} \ | \ A\subseteq X\right\}.\]
\end{definition}

Several of the previous examples, including order ideals, are also convex geometries. 
In particular, we show below that the set of interval-closed sets of a poset (see Section~\ref{sec:ic}) is a convex geometry. 

\begin{proposition}
\label{prop:ic_convgeom}
Let $P$ be a finite poset and $\I(P)$ the set of interval-closed sets of $P$. 
$\I(P)$ is a convex geometry.
\end{proposition}

\begin{proof}
The empty set and $P$ are both convex sets, so condition (1) of Definition~\ref{def:cxgeom} is satisfied. The intersection of any two interval-closed sets is also interval-closed, satisfying condition (2). Condition (3) is satisfied, since, given $I\in\I(P)$, $I\neq P$, and $x\in P$, if $I\cup\{x\}\notin\I(P)$, then there must be $y_1\in P$ such that the interval-closed set induced by $I\cup\{x\}$ contains $y_1$. But then either $I\cup\{y_1\}\in\I(P)$ or there exists $y_2\in P$ such that the interval-closed set induced by $I\cup\{y_1\}$ contains $y_2$. Since the poset is finite and the partial order is acyclic, after a finite number of steps we will find $y_i$ such that $I\cup\{y_i\}\in\I(P)$. 
\end{proof}

\begin{remark} \rm
One might speculate that the set of convex geometries $\mathcal{L}$ that are not a toggle-disjoint Cartesian product would be inductively toggle-alternating.
Proposition~\ref{prop:ic_convgeom} shows this is not the case, since in Remark~\ref{remark:icnotpleasant} we exhibited a poset $P$ for which 
$\mathcal{IC}(P)$ is toggle-alternating, but not inductively toggle-alternating. 
So we cannot use Theorem~\ref{thm:toggpstructure} to obtain a structure theorem for all convex geometries. It may be possible to characterize an inductively toggle-alternating family of convex geometries, as we did for interval-closed sets in Corollary~\ref{cor:icsets}; we leave this to future work.
\end{remark}

In the next section, we define a convex closure map and relate it to rowmotion.

\section{Cover-closure and rowmotion}
\label{sec:row}
In this section, we define a map we call \emph{cover-closure} on any family of subsets $\mathcal{L}\subseteq 2^E$ 
with a closure operator. We show in Lemma~\ref{lem:ccrow} that if $\mathcal{L}$ is the set of order ideals of a  poset,  cover-closure is exactly the rowmotion action of~\cite{fonderflaass,prorow}. We prove in Theorem~\ref{thm:row} that if 
cover-closure is bijective, 
then $\mathcal{L}$ must be in bijection with the set of order ideals for some poset, thus rowmotion is the only bijective cover-closure map.

\begin{definition}
\label{def:closure}
Given a ground set $E$, a \emph{closure operator} is a function $\tau:2^E\rightarrow 2^E$ that satisfies the following for all $X,Y\subseteq E$:
\begin{enumerate}
\item $X\subseteq\tau(X)$ (\emph{extensivity});
\item $X\subseteq Y$ implies $\tau(X)\subseteq\tau(Y)$ (\emph{monotonicity}); and
\item $\tau(\tau(X))=\tau(X)$ (\emph{idempotence}).
\end{enumerate}
A \emph{closed set} with respect to $\tau$ is a set $X\subseteq E$ for which $\tau(X)=X$.
\end{definition}

\begin{example} \rm
\label{ex:oi_closure}
A standard closure operator on a poset $P$, is given by $\tau(X)$, for $X\in 2^P$, equals the intersection of all the order ideals in $J(P)$ containing $X$.
Then the set of closed sets of $\tau$ is precisely $J(P)$.
\end{example}

\begin{remark} \rm
Note a convex geometry may be alternatively characterized as a set with a closure operator $\tau$ satisfying the additional property (called the \emph{anti-exchange axiom}) that if neither $e$ nor $f$ belong to $\tau(X)$, but $f$ belongs to $\tau(X\cup \{e\})$, then $e$ does not belong to $\tau(X\cup\{f\})$. See~\cite{EdelmanJamison}.
\end{remark}

Using the closure operator of Definition~\ref{def:closure}, we define the following map.
\begin{definition}
\label{def:cc}
Let $E$ be a  set with closure operator $\tau$ whose set of closed sets is $\mathcal{L}\subseteq 2^E$. 
For $X\in 2^E$, %let $\cov(X)\subseteq E\setminus X$ be  
define $\cov(X)=\{e\in E\setminus X \ | \ X\cup\{e\}\in\mathcal{L} \}$.
%the maximal subset of $E\setminus X$ such that $\forall e\in \cov(X)$, $X\cup\{e\}\in\mathcal{L}$. 
Call $\cov(X)$ the set of \emph{covers} of $X$.
Then we define \emph{cover-closure} $\cc:\mathcal{L} \rightarrow \mathcal{L}$ as $\cc(X)=\tau(\cov(X))$.
\end{definition}

\begin{remark} \rm
Note that if $X\in\mathcal{L}$, $\cov(X)$ is the set of labels of the edges of the covers of $X$ in the toggle poset $\mathcal{P}_{\mathcal{L}}$ of Definition~\ref{def:Lposet}.
\end{remark}

\begin{example} \rm
\label{ex:notbij}
Let $E=\left\{1,2,3,4\right\}$. Define $\tau$  as the closure operator on $E$ with set of closed sets  $\mathcal{L}=\{\emptyset, \{1\}, \{2\}, \{3\}, \{4\},  \{1,2\}, \{1,3\}, \{2,3\} ,\{2,4\}, \{3,4\},  \{1,2,3\},$ $\{2,3,4\}, \{1,2,3,4\}\}$
and such that $\tau(\{1,4\})=\tau(\{1,2,4\})=\tau(\{1,3,4\})$ $=E$.  
Then cover-closure $\cc$ on $\mathcal{L}$ is given in Figure~\ref{fig:cc}.
\end{example}

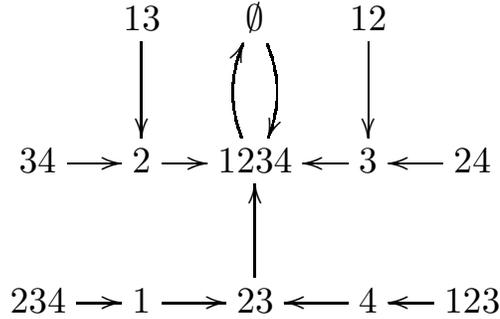
\begin{figure}[htbp]
\begin{center}
\scalebox{1.4}{
$$ \xymatrix @-1.2pc {
& 13 \ar[dd] & \emptyset  \ar@/^/[dd]  & \ar[dd] 12 & \\
&&&&\\
34 \ar[r] & 2\ar[r] & 1234 \ar@/^/[uu] & \ar[l] 3 & \ar[l] 24 \\
&&&&\\
234 \ar[r] & 1\ar[r] & 23\ar[uu] & 4\ar[l] & \ar[l] 123
}$$}
\end{center}
\caption{The cover-closure map $\cc$ on $\mathcal{L}$, where $\mathcal{L}$ is as in Example~\ref{ex:notbij}.}
\label{fig:cc}
\end{figure}

The following lemma shows that cover-closure on the set of order ideals $J(P)$ is the \emph{rowmotion} action of~\cite{prorow}.
\begin{lemma}
\label{lem:ccrow}
If $\mathcal{L}=J(P)$ for some finite poset $P$ and $\tau(X)$ is the intersection of all the order ideals in $J(P)$ containing $X$, then cover-closure $\cc:\mathcal{L}\rightarrow \mathcal{L}$ is rowmotion on $J(P)$.
\end{lemma}

\begin{proof}
Rowmotion on an order ideal $I$ of a poset $P$ is defined as the order ideal generated by the minimal elements of $P\setminus I$. (See Definition~\ref{def:row}.) Thus if $\mathcal{L}=J(P)$ and $\tau(X)$ is the intersection of all order ideals containing $X$, we have that cover-closure $\cc$ is exactly rowmotion.
\end{proof}

The following theorem characterizes bijective cover-closure maps. 

\begin{theorem}
\label{thm:row}
Given a finite ground set $E$ and a closure operator $\tau$ with set of closed sets $\mathcal{L}\subseteq 2^E$,
cover-closure $\cc=\tau\circ\cov:\mathcal{L}\rightarrow\mathcal{L}$ is bijective if and only if 
$\mathcal{L}$ is isomorphic to the set of order ideals $J(P)$ for some poset~$P$.
\end{theorem}

\begin{proof}
Let $E$ be a finite ground 
set with closure operator $\tau$ and set of closed sets $\mathcal{L}\subseteq 2^E$.
By Lemma~\ref{lem:ccrow}, if $\mathcal{L}=J(P)$, then cover-closure is rowmotion, which is bijective. It is left to show that if cover-closure on $\mathcal{L}$ is bijective, then there exists a poset $P$ such that $\mathcal{L}$ is isomorphic to the set of order ideals $J(P)$.

Suppose $\cc:\mathcal{L}\rightarrow\mathcal{L}$ is bijective. We first show $\mathcal{L}$ is isomorphic to a convex geometry (see Definition~\ref{def:cxgeom}).
First, $\tau(E)=E$ by extensivity, so $E\in\mathcal{L}$. Now suppose $\tau(\emptyset)=A\neq\emptyset$. Then $A\subseteq X$ for all $X\in\mathcal{L}$ by monotonicity of $\tau$, therefore no element of $A$ is in the essentialized ground set $E'$. Thus we work in the essentialization $(\mathcal{L}',E')$ with closure operator $\tau':2^{E'}\rightarrow 2^{E'}$ such that $\mathcal{L}'$ is the set of closed sets of $\tau'$. 
Thus, $\tau'$ satisfies $\tau'(\emptyset)=\emptyset$.
So $\emptyset$ is in $\mathcal{L}'$, $\mathcal{L}'$ satisfies condition (1) of Definition~\ref{def:cxgeom}, 
and $\mathcal{L}'$ is isomorphic to $\mathcal{L}$, which implies their toggle posets $\mathcal{P}_{\mathcal{L}}$ and $\mathcal{P}_{\mathcal{L}'}$ are isomorphic.

From now on, we consider $\cc:\mathcal{L}'\rightarrow\mathcal{L}'$ to be $\tau'\circ\cov$, 
which is a bijection since $\cc$ is a bijection on $\mathcal{L}$.
Since $\mathcal{L}'$ is the set of closed sets of $\tau'$, the toggle poset $\mathcal{P}_{\mathcal{L}'}$ must be a lattice, where the meet and join are given as $X\wedge Y=\tau'(X\cap Y)=X\cap Y$ and $X\vee Y = \tau'(X\cup Y)$. Note that $\tau'(X\cap Y)=X\cap Y$ for $X,Y\in\mathcal{L}'$ holds by the following argument: $X\cap Y\subseteq \tau'(X\cap Y)$ by extensivity, while $\tau'(X\cap Y)\subseteq \tau'(X)$ and $\tau'(X\cap Y)\subseteq \tau'(Y)$ by monotonicity. $\tau'(X)=X$ and $\tau'(Y)=Y$ since $X,Y\in\mathcal{L}'$, so $\tau'(X\cap Y)\subseteq X\cap Y$; thus $\tau'(X\cap Y)=X\cap Y$. 
Therefore, $\mathcal{L}'$ satisfies condition $(2)$ of Definition~\ref{def:cxgeom}. 

Since cover-closure is a bijection, there must be a one-to-one correspondence between $X$ and $\cov(X)$ for all $X\in\mathcal{L}'$. In particular, the set $\cov(X)$ is unique for each $X$, otherwise $\cc$ would map two distinct elements of $\mathcal{L}'$ to the same image.
Now for each $X\in\mathcal{L}'$ such that $X\neq E'$, there exists $e\in E'\setminus X$ such that $X\cup\{e\}\in \mathcal{L}'$, otherwise $\cov(X)=\emptyset$ and $\cc(X)=\emptyset=\cc(E')$ which would be a contradiction. So condition (3) of Definition~\ref{def:cxgeom} is satisfied, and $\mathcal{L}'$ is a convex geometry.

Now to show $\mathcal{L}'$ is, furthermore, the set of order ideals $J(P)$ for some poset~$P$, we need only show that for $X,Y\in \mathcal{L}'$, $\tau'(X\cup Y)=X\cup Y$, that is, $\mathcal{L}'$ is closed under set union. This will imply that  $\mathcal{P}_{\mathcal{L}'}$ is a distributive lattice, since the join and meet are given by $X\vee Y = X\cup Y$ and $X\wedge Y = X\cap Y$. Then by the fundamental theorem of finite distributive lattices (see e.g.~\cite{Stanley_1}), $\mathcal{L}'$ will be the set of order ideals $J(P)$ for some poset~$P$.

If there is some set $Y$ such that $\tau'(Y)=Z\cup\{e\}$ for some $Z\in \mathcal{L}'$, $e \in E'\setminus Z$,
then $e$ must be in $Y$, or else $\tau'(Y)$ would be at most $Z$, since $Z$ is a closed set smaller than $Z\cup\{e\}$ containing $Y$. 
Thus, given $X\in \mathcal{L}'$, if $\cc(X)=Z\cup\{e\}$ for some $Z\in \mathcal{L}'$, $e \in E'\setminus Z$, then $e$ is in $\cov(X)$, since $\cc(X)=\tau'(\cov(X)))$. 
Therefore, $\cov(X)$ contains all the elements that we can remove singly from $\cc(X)$ and remain in $\mathcal{L}'$. 
Denote the set of such elements as $\rem(X):=\left\{e\in X \mid X\setminus\{e\}\in\mathcal{L}'\right\}$. So we have $\rem(\cc(X))\subseteq \cov(X)$ for all $X\in\mathcal{L}'$. 
We wish to show $\rem(\cc(X)) = \cov(X)$ 

Now $\displaystyle\sum_{X\in\mathcal{L}'}\cov(X)$ equals the number of edges in the Hasse diagram of $\mathcal{P}_{\mathcal{L}'}$. Since $\cc$ is a bijection, $\displaystyle\sum_{X\in\mathcal{L}'}\rem(\cc(X)) = \displaystyle\sum_{X\in\mathcal{L}'}\rem(X)$, which also equals the number of edges in the Hasse diagram of $\mathcal{P}_{\mathcal{L}'}$. So $\displaystyle\sum_{X\in\mathcal{L}'}\rem(\cc(X)) = \displaystyle\sum_{X\in\mathcal{L}'}\cov(X)$, but $\rem(\cc(X))\subseteq \cov(X)$ for all $X\in\mathcal{L}'$ and $\mathcal{L}'$ is a finite set, so it must be that $\rem(\cc(X))=\cov(X)$ for all $X\in\mathcal{L}'$.

Consider the dual $\widehat{\mathcal{P}}_{\mathcal{L}'}$ of the lattice $\mathcal{P}_{\mathcal{L}'}$. Let $\widehat{\mathcal{L}'}:=\left\{E'\setminus X \mid X\in\mathcal{L}'\right\}$. Then  $\widehat{\mathcal{L}'}$ is the set of closed sets for the dual closure operator $\widehat{\tau'}$. Cover-closure may be defined on $\widehat{\mathcal{L}'}$
as $\widehat{\cc}:\widehat{\mathcal{L}'}\rightarrow \widehat{\mathcal{L}'}$, $\widehat{\cc}=\widehat{\tau'}\circ\widehat{\cov}$ where for $E'\setminus X\in\widehat{\mathcal{L}'}$, $\widehat{\cov}(E'\setminus X)=\rem(X)=\cov(\cc^{-1}(X))$, so $\widehat{\cc}$ 
is bijective. By the same arguments as for $\mathcal{L}'$ above, $\widehat{\mathcal{L}'}$ is a convex geometry, and so is closed under set intersection. Therefore, ${\mathcal{L}'}$ is closed under set union, and so $\mathcal{P}_{\mathcal{L}'}$ is a distributive lattice. Then since $\mathcal{P}_{\mathcal{L}'}$ and $\mathcal{P}_{\mathcal{L}}$ are isomorphic, $\mathcal{P}_{\mathcal{L}}$ is also a distributive lattice, and the theorem follows.
\end{proof}

\begin{corollary}
\label{cor:rowcc}
Rowmotion is the only bijective cover-closure map on a finite ground set.
\end{corollary}

\acknowledgements{The author thanks David Einstein for helpful conversations on Theorem~\ref{thm:row}, Vic Reiner for suggesting the extension of the toggle group to convex geometries, and Nathan Williams for introducing her to the toggle group. The author also thanks the developers of \verb|SageMath|~\cite{sage} and \verb|GAP|~\cite{gap} for sharing their code by which some of this research was conducted, the developers of \verb|CoCalc|~\cite{SMC} for making \verb|SageMath| more accessible, and the anonymous referees for helpful comments.}

\appendix
\section{SageMath code for computing toggle group structure}
\label{sec:appendix}

We give below the \verb|SageMath|~\cite{sage} code we wrote to check all the base cases of this paper. (This code should work in either \verb|CoCalc|~\cite{SMC} or your local installation of \verb|SageMath|~\cite{sage}.) The input L (denoted $\mathcal{L}$ in this paper) is a list of subsets of the ground set $\mathrm{E}=[1,2,3,4]$. To compute the toggle group on a larger ground set, change the number $5$ in the third line of the first function to an appropriate larger integer. 

For each toggle $t_e$, $e\in$ E, the first function iterates through L and notes whenever it finds two subsets that differ by a single toggle. The two subsets are at indices i and j in the list L (and the indexing starts at 0), so the transposition $(i+1,j+1)$ is added to the permutation for the toggle $t_e$. The function \verb|toggle_group_permutations| continues in this way and returns a list of permutations corresponding to the toggles. The second function, \verb|toggle_group_structure|, calls \verb|GAP|~\cite{gap} from within \verb|SageMath| to construct the group and output its structure description (given in \verb|GAP| syntax).

\begin{lstlisting}[language=Python]
def toggle_group_permutations(L):
    togglepermutations = []
    for tog in range(1,5):
        togperm = []
        for i in range(len(L)):
            for j in range(len(L)):
                if L[j].issubset(L[i]):
                    if list(L[i].symmetric_difference(L[j])) = = [tog]:
                        togperm.append((i+1,j+1))
        togglepermutations.append(Permutation(togperm).cycle_string())
    return togglepermutations

def toggle_group_structure(L):
    groupgens = toggle_group_permutations(L)
    mygroup = PermutationGroup(groupgens)
    return mygroup.structure_description()
\end{lstlisting}

\medskip

As an example, to verify Example~\ref{ex:not_sn_an}, run the following: 

\begin{lstlisting}[language=Python]
L = [set({}),{1},{1,2},{1,2,3},{1,2,3,4},{2,3,4},{3,4},{4}]
toggle_group_permutations(L)
toggle_group_structure(L)
\end{lstlisting}

The output will be:

\medskip

\noindent
$ [ \ \text{\textquotesingle}(1,2)(5,6)\text{\textquotesingle}, \ \text{\textquotesingle}(2,3)(6,7)\text{\textquotesingle}, \ \text{\textquotesingle}(3,4)(7,8)\text{\textquotesingle}, \ \text{\textquotesingle}(1,8)(4,5)\text{\textquotesingle} \ ] $

\medskip

\noindent
$\text{\textquotesingle} \ ( \ ( \ (\textrm{C}2 \ \times \ \textrm{D}4) \ : \ \textrm{C}2) \ : \ \textrm{C}3) \ : \ \textrm{C}2 \ \text{\textquotesingle}$

\medskip

\begin{thebibliography}{99}
%\bibliographystyle{plainnat}
\bibitem[DAC]{DACAIM}
Dynamical algebraic combinatorics workshop, American Institute of Mathematics, San Jose, California, 
March 23--27, 2015,
\url{http://aimath.org/pastworkshops/dynalgcomb.html}.
\bibitem[CF95]{fonderflaass}
P.\ Cameron and D.\ Fon-Der-Flaass. Orbits of antichains revisited, \emph{European J.\ Combin.} {\bf 16} (1995), no.\ 6, 545--554.
\bibitem[CoCalc]{SMC}
SageMath, Inc., CoCalc Online Computational Mathematics, (2017), \url{https://cocalc.com/}.
\bibitem[DPS17]{DilksPechenikStriker} K.\ Dilks, O.\ Pechenik, and J.\ Striker. Resonance in orbits of plane partitions and increasing tableaux, \emph{J.\ Combin. Theory Ser.\ A} {\bf 148} (2017), 244--274.
\bibitem[EJ85]{EdelmanJamison}
P.\ Edelman and R.\ Jamison. The theory of convex geometries, 
\emph{Geom.\ Dedicata} \textbf{19} (1985), no.\ 3, 247--270. 
\bibitem[EP14]{EinsteinPropp}
D.\ Einstein and J.\ Propp. Piecewise-linear and birational toggling, \emph{Discrete Math.\ Theor.\ Comput.\ Sci.\ Proc}.\ (FPSAC 2014) {\bf AT}, 513--524.
\bibitem[EFG+16]{EFGJMPRS}
D.\ Einstein, M.\ Farber, E.\ Gunawan, M.\ Joseph, M.\ Macauley, J.\ Propp, and S.\ Rubinstein-Salzedo. Noncrossing partitions, toggles, and homomesies, \emph{Electron.\ J.\ Combin.} {\bf 23} (2016), no.\ 3, Paper 3.52, 26 pp. 
\bibitem[GAP]{gap}
The GAP~Group, \emph{GAP -- Groups, Algorithms, and Programming, 
  Version 4.7.8}. (2015) \url{http://www.gap-system.org}.
\bibitem[GR16]{Darij_Tom1}
D.\ Grinberg and T.\ Roby. Iterative properties of birational rowmotion I: generalities and skeletal posets, \emph{Electron.\ J.\ Combin.} {\bf 23} (2016), no.\ 1.
\bibitem[GR15]{Darij_Tom2}
D.\ Grinberg and T.\ Roby. Iterative properties of birational rowmotion II: rectangles and triangles, \emph{Electron.\ J.\ Combin.} {\bf 22} (2015), no.\ 3.
\bibitem[Hu00]{humphreys} 
J.\ Humphreys. Reflection Groups and Coxeter Groups. Cambridge: Cambridge Univ.\ Press (2000).
\bibitem[JR18]{MJoseph}
M.\ Joseph and T.\ Roby. Toggling independent sets of a path graph. \emph{Electron.\ J.\ Combin.} \textbf{25} (2018), no.\ 1, Paper 1.18, 31 pp.
\bibitem[KSW15]{kubik}
B.\ Kubik and S.\ Sather-Wagstaff. Path ideals of weighted graphs, \emph{J.\ Pure Appl.\ Algebra} \textbf{219} (2015), no.\ 9, 3889--3912. 
\bibitem[LPW09]{markovmixing} D.\ Levin, Y.\ Peres, and E.\ Wilmer. Markov Chains and Mixing Times. Providence R I: American Math. Soc.\ (2009).
\bibitem[PSW13]{paulsen}
C.\ Paulsen and S.\ Sather-Wagstaff. Edge ideals of weighted graphs, \emph{J.\ Algebra Appl.} \textbf{12} (2013), no.~5, 24 pp.
\bibitem[PR15]{homomesy}
J.\ Propp and T.\ Roby. Homomesy in products of two chains, \emph{Electron.\ J.\ Combin.} {\bf 22} (2015), no.\ 3.
\bibitem[RSW04]{CSP}
V.\ Reiner, D.\ Stanton, and D.\ White. The cyclic sieving phenomenon, \emph{J.\ Combin.\ Theory Ser.\ A} {\bf 108} (2004), no.\ 1, 17--50.
\bibitem[Sage]{sage} SageMath, the Sage Mathematics Software System (Version 6.5), The Sage Developers, (2015),
\url{http://www.sagemath.org}.
\bibitem[St12]{Stanley_1}
R.\ Stanley. Enumerative Combinatorics.\ Volume 1,\ Second edition.\ Cambridge Studies in Advanced Mathematics, {\bf 49}, Cambridge University Press, Cambridge, 2012.
\bibitem[St11]{StrikerPoset}
J.\ Striker. A unifying poset perspective on alternating sign matrices, plane partitions, Catalan objects, tournaments, and tableaux, \emph{Adv.\ Appl.\ Math.} {\bf 46} (2011), no.\ 1--4, 583--609.
\bibitem[SW12]{prorow}
J.\ Striker and N.\ Williams. Promotion and rowmotion, \emph{Eur.\ J.\ Combin.} {\bf 33} (2012), no.\ 8, 1919--1942.
\end{thebibliography}
\end{document}